\documentclass[journal]{IEEEtran}%
\usepackage{graphicx}
\usepackage{amsmath}
\usepackage{graphics}
\usepackage{amsfonts}
\usepackage{amsthm}
\usepackage{amssymb}%
\providecommand{\U}[1]{\protect\rule{.1in}{.1in}}
\newtheorem{theorem}{Theorem}
\newtheorem{remark}{Remark}

\newtheorem{lemma}{Lemma}
\newtheorem{property}{Property}
\ifCLASSINFOpdf
\else
\fi
\begin{document}

\title{A Cyclic Coordinate Descent Algorithm for $l_q$  Regularization
\thanks{This work was partially supported by the National 973 Programs (Grant No. 2013CB329404),
the Key Program of National Natural Science Foundation of China (Grants No.
11131006), the National Natural Science Foundations of China (Grants No. 11001227, 11171272),
NSF Grants NSF DMS-1349855 and DMS-1317602.}
}
\author{Jinshan Zeng,~Zhimin Peng,~Shaobo Lin, Zongben Xu
%, Shaobo Lin
%, Wotao Yin and~Zongben Xu
\thanks{ J.S. Zeng is with the Institute for Information and System
Sciences, School of Mathematics and Statistics, Xi'an Jiaotong University,
Xi'an 710049, and Beijing Center for Mathematics and Information Interdisciplinary Sciences (BCMIIS), Beijing, 100048, China.
Z.M. Peng is with the Department of Mathematics, University of California, Los Angeles (UCLA), Los Angeles, CA 90095, United States.
S.B. Lin and Zongben Xu are with
the Institute for Information and System Sciences, School of Mathematics and Statistics, Xi'an Jiaotong University, Xi'an 710049, P R China.
% is with the department of mathematics, University of California, Los Angeles (UCLA), Los Angeles, CA 90095, United States.
%S.B. Lin is with the Institute for Information and System
%Sciences, School of Mathematics and Statistics, Xi'an Jiaotong University,
%Xi'an 710049, P R China.
(email: jsh.zeng@gmail.com, zhimin.peng@math.ucla.edu, sblin1983@gmail.com, zbxu@mail.xjtu.edu.cn).
%, zbxu@mail.xjtu.edu.cn
%, sblin1983@gmail.com).
%wotaoyin@math.ucla.edu

%sblin1983@gmail.com, wotaoyin@math.ucla.edu, zbxu@mail.xjtu.edu.cn

$*$ Corresponding author: Shaobo Lin (sblin1983@gmail.com).
}}
\maketitle

\begin{abstract}
%\boldmath
In recent studies on sparse modeling, $l_q$ ($0<q<1$) regularization has received considerable attention due to
its superiorities on sparsity-inducing and bias reduction over the $l_1$ regularization.
In this paper, we propose a cyclic coordinate descent (CCD) algorithm for $l_q$ regularization.
Our main result states that the CCD algorithm converges globally to a stationary point as long as the stepsize is less than a positive constant. Furthermore, we demonstrate that the CCD algorithm converges to a local minimizer under certain additional conditions.
Our numerical experiments demonstrate the efficiency of the CCD algorithm.
\newline

\end{abstract}

%paper title
%can use linebreaks \\ within to get better formatting as desired

%author names and IEEE memberships
%note positions of commas and nonbreaking spaces ( ~ ) LaTeX will not break
%a structure at a ~ so this keeps an author's name from being broken across
%two lines.
%use \thanks{} to gain access to the first footnote area
%a separate \thanks must be used for each paragraph as LaTeX2e's \thanks
%was not built to handle multiple paragraphs

%note the % following the last \IEEEmembership and also \thanks -
%these prevent an unwanted space from occurring between the last author name
%and the end of the author line. i.e., if you had this:
%\author{....lastname \thanks{...} \thanks{...} }
%^------------^------------^----Do not want these spaces!
%a space would be appended to the last name and could cause every name on that
%line to be shifted left slightly. This is one of those "LaTeX things". For
%instance, "\textbf{A} \textbf{B}" will typeset as "A B" not "AB". To get
%"AB" then you have to do: "\textbf{A}\textbf{B}"
%\thanks is no different in this regard, so shield the last } of each \thanks
%that ends a line with a % and do not let a space in before the next \thanks.
%Spaces after \IEEEmembership other than the last one are OK (and needed) as
%you are supposed to have spaces between the names. For what it is worth,
%this is a minor point as most people would not even notice if the said evil
%space somehow managed to creep in.

%The paper headers
\markboth{ ~}{Shell \MakeLowercase{\textit{et al.}}: Bare Demo of IEEEtran.cls for Journals}
%The only time the second header will appear is for the odd numbered pages
%after the title page when using the twoside option.
%*** Note that you probably will NOT want to include the author's ***
%*** name in the headers of peer review papers.                   ***
%You can use \ifCLASSOPTIONpeerreview for conditional compilation here if
%you desire.

%If you want to put a publisher's ID mark on the page you can do it like
%this:
%Remember, if you use this you must call \IEEEpubidadjcol in the second
%column for its text to clear the IEEEpubid mark.

%use for special paper notices

%make the title area

%IEEEtran.cls defaults to using nonbold math in the Abstract.
%This preserves the distinction between vectors and scalars. However,
%if the journal you are submitting to favors bold math in the abstract,
%then you can use LaTeX's standard command \boldmath at the very startF
%of the abstract to achieve this. Many IEEE journals frown on math
%in the abstract anyway.

%Note that keywords are not normally used for peerreview papers.

%For peer review papers, you can put extra information on the cover
%page as needed:
%\ifCLASSOPTIONpeerreview
%\begin{center} \bfseries EDICS Category: 3-BBND \end{center}
%\fi
%For peerreview papers, this IEEEtran command inserts a page break and
%creates the second title. It will be ignored for other modes.
\IEEEpeerreviewmaketitle

\section{Introduction}

%The very first letter is a 2 line initial drop letter followed
%by the rest of the first word in caps.
%form to use if the first word consists of a single letter:
%\IEEEPARstart{A}{demo} file is ....
%form to use if you need the single drop letter followed by
%normal text (unknown if ever used by IEEE):
%\IEEEPARstart{A}{}demo file is ....
%Some journals put the first two words in caps:
%\IEEEPARstart{T}{his demo} file is ....
%Here we have the typical use of a "T" for an initial drop letter
%and "HIS" in caps to complete the first word.
%\IEEEPARstart{T}{his} demo file is intended to serve as a ``starter file''
%for IEEE journal papers produced under \LaTeX\ using
%IEEEtran.cls version 1.7 and later.
%% You must have at least 2 lines in the paragraph with the drop letter
%% (should never be an issue)
%I wish you the best of success.

%\hfill January 11, 2007

Recently, the sparse vector recovery problems have attracted lots of attention in both scientific research and engineering practice (\cite{Donoho06}-\cite{SARZeng2012}).
Typical applications include compressed sensing \cite{Donoho06}, {\cite{Candes06}},
statistical regression {\cite{TibshiraniVS}},
visual coding {\cite{OlshausenVC}}, signal processing {\cite{Combettes2005ProxInSP}},
machine learning {\cite{1-SVMZhu2003}},
magnetic resonance imaging (MRI) {\cite{Lustig2008SparseMRI}}
and microwave imaging {\cite{SARZeng2012}}, {\cite{Zeng2013AccSAR}}.
In a general setup, an unknown sparse vector $x\in\mathbf{R}^{N}$ is
reconstructed from measurements
\begin{equation}
y=Ax, \label{ObsModel}%
\end{equation}
or more generally, from
\begin{equation}
y=Ax+\epsilon, \label{ObsMode2}%
\end{equation}
where $y\in\mathbf{R}^{m}$, $A\in\mathbf{R}^{m\times N}$ (commonly, $m<N$) is
a measurement matrix and $\epsilon$ represents the noise. The problem can be
modeled as the $l_{0}$ regularization problem
\begin{equation}
\min_{x\in\mathbf{R}^{N}}\left\{\frac{1}{2}{\Vert Ax-y\Vert}_{2}^{2}+\lambda{\Vert x\Vert}%
_{0}\right\},
\label{L0Reg}%
\end{equation}
where $\Vert x\Vert_{0}$, formally called the $l_{0}$ norm, denotes the
number of nonzero components of $x$, and $\lambda>0$ is a regularization parameter.
However, due to its NP-hardness \cite{Natarjan95}, $l_{0}$ regularization is generally intractable.

%These problems can be mathematically modeled as the following $l_0$-norm regularized optimization problems
%\begin{equation}
%\min_{x\in \mathbf{R}^N} \left\{ F(x) + \lambda \|x\|_0 \right\},
%\label{L0Reg}
%\end{equation}
%where $F: \mathbf{R}^N \rightarrow [0,\infty)$ is a proper lower-semicontinuous function,
%$\|x\|_0$, formally called the $l_0$ quasi-norm, denotes the number of nonzero components of $x$ and $\lambda>0$ is a regularization parameter.
%However, the optimization problem (\ref{L0Reg}) is generally intractable to solve, due to the discontinuity and non-convexity of the $l_0$ quasi-norm.

In order to overcome such difficulty,
many continuous penalties were proposed to substitute the $l_0$ norm by the following optimization problem
\begin{equation}
\min_{x\in \mathbf{R}^N} \left\{\frac{1}{2}{\Vert Ax-y\Vert}_{2}^{2} + \lambda \Phi(x)\right\},
\label{NonSCS}
\end{equation}
where $\Phi(x)$ is a separable, continuous penalty with $\Phi(x) = \sum_{i=1}^N \phi(x_i)$, and $x=(x_1,\cdots,x_N)^T$.
One of the most important cases is the $l_1$ norm, i.e., $\Phi(x) = \|x\|_1=\sum_{i=1}^N |x_i|$.
The $l_1$ norm is convex and thus, the corresponding $l_1$ norm regularized convex optimization problem can be efficiently solved.
Because of this, the $l_1$ norm gets its popularity and has been accepted as a very useful method for the modeling of the sparsity problems.
Nevertheless, the $l_1$ norm may not induce further sparsity when applied to certain applications {\cite{Candes2008RL1}}, {\cite{Chartrand2007}}, {\cite{Chartrand2008}}.
Alternatively, many non-convex penalties were introduced as relaxations of the $l_0$ norm.
Among these, the $l_q$ norm with $0<q<1$ (not an actual norm when $0<q<1$), i.e., $\|x\|_q = (\sum_{i=1}^N |x_i|^q)^{1/q}$ is one of the most typical subsitutions.
Compared with the $l_1$ norm, the $l_q$ norm can usually induce better sparsity and reduce the bias while the corresponding non-convex regularized optimization problems are generally more difficult to solve.

Several classes of algorithms have been developed to solve the non-convex regularized optimization problem (\ref{NonSCS}). These algorithms include half-quadratic (HQ) algorithm {\cite{GRHQ1992}}, {\cite{GYHQ1995}}, iteratively reweighted algorithm {\cite{Candes2008RL1}}, {\cite{FOCUSS}},
difference of convex functions algorithm (DC programming) {\cite{DCProgGasso2009}},
iterative thresholding algorithm {\cite{BrediesNonconvex}}, {\cite{L1/2TNN}}, {\cite{ZengIJT2014}},
and cyclic coordinate descent (CCD) algorithm {\cite{CDTseng2001}}, {\cite{lqCD-Marjanovic-2014}}.

The first class is the half-quadratic (HQ) algorithm {\cite{GRHQ1992}}, {\cite{GYHQ1995}}.
The basic idea of HQ algorithm is to formulate the original objective function as an infimum of a family of augmented functions via introducing a dual variable, and then minimize the augmented function along the primal and dual variables in an alternate fashion. However, HQ algorithms can be efficient only when both subproblems are easily solved (particularly, when both subproblems have the closed-form solutions).
The second class is the iteratively reweighted algorithm which includes the iteratively reweighted least squares minimization (IRLS) {\cite{FOCUSS}}, {\cite{IRAChartrand_Yin2008}}, {\cite{DaubechiesIRLS}},
and iteratively reweighted $l_1$-minimization (IRL1) {\cite{Candes2008RL1}}.
More specifically, the IRLS algorithm solves a sequence of weighted least squares problems, which can be viewed as some approximate problems to the original optimization problem.
Similarly, the IRL1 algorithm solves a sequence of non-smooth weighted $l_1$-minimization problems,
and hence can be seen as the non-smooth counterpart to the IRLS algorithm.
Nevertheless, the iteratively reweighted algorithms can be only efficient when applied to such non-convex regularization problems whose non-convex penalty can be well approximated by the quadratic function or the weighted $l_1$ norm function.

The third class is the difference of convex functions algorithm (DC programming) {\cite{DCProgGasso2009}}, which is also called
Multi-Stage (MS) convex relaxation {\cite{ZhangMS2010}}.
The key idea of DC programming is to consider a proper decomposition of the objective function.
More specifically, it converts the non-convex penalized problem into a convex reweighted $l_1$ minimization problem (called primal problem) and another convex problem (called dual problem), and then iteratively optimizes the primal and dual problems {\cite{DCProgGasso2009}}.
Hence, it can only be applied to a certain family of non-convex penalties that can be decomposed as a difference of convex functions.
The fourth class is the iterative thresholding algorithm {\cite{L1/2TNN}}, {\cite{ZengIJT2014}}, {\cite{Daubechies2004Soft}}, {\cite{Blumensath08}}, which fits the framework of the forward-backward splitting algorithm {\cite{Attouch2013}} and the framework of the generalized gradient projection algorithm {\cite{BrediesNonconvex}}.
Intuitively, the iterative thresholding algorithm can be seen as a procedure of
Landweber iteration projected by a certain thresholding operator.
%Thus, the thresholding operator plays a key role in the iterative thresholding algorithm.
%For some special non-convex penalties such as SCAD, MCP, LSP and $l_q$-norms with $q=1/2, 2/3$,
%the associated thresholding operators can be expressed analytically {\cite{L1/2TNN}}, {\cite{L2/3Cao2013}}, {\cite{YeNonconvex}}.
Compared with the other types of non-convex algorithms, the iterative thresholding algorithm can be easily implemented and has relatively lower computational complexity for large scale problems {\cite{SARZeng2012}}, {\cite{Zeng2013AccSAR}}, {\cite{Qian20011}}.
However, the iterative thresholding algorithm can only be effectively applied to models with some particular structures.

The last class is the cyclic coordinate descent (CCD) algorithm.
Basically, CCD algorithm is a coordinate descent algorithm with the cyclic coordinate updating rule.
In {\cite{l1CDFriedman2007}}, a CCD algorithm was implemented for solving the $l_1$ regularization problem.
Its convergence can be shown by referring to {\cite{CDTseng2001}}.
In {\cite{ncxCDMazumder2007}}, a CCD algorithm was proposed for a class of non-convex penalized least squares problems.
However, both {\cite{ncxCDMazumder2007}} and {\cite{CDTseng2001}} do not consider the CCD algorithm for $l_q$ regularization problem.
Recently, Marjanovic and Solo {\cite{lqCD-Marjanovic-2014}} proposed a cyclic descent algorithm (called $l_q$CD) for the normalized $l_q$ regularization problem where the columns of $A$ are normalized with the unit $l_2$ norm, i.e., $\|A_i\|_2 = 1$, $i=1,2,\dots, N$, where $A_i$ is the $i$-th column of $A$.
They proved the subsequential convergence and furthered the convergence to a local minimizer under the scalable restricted isometry property (SRIP)
%were justified
in {\cite{lqCD-Marjanovic-2014}}.
According to {\cite{lqCD-Marjanovic-2014}}, the column-normalization requirement is crucial for the convergence analysis of $l_q$CD algorithm.
However, such requirement may limit the applicability of the $l_q$CD algorithm, and also
%However, such requirement may
will introduce some additional computational complexity.
%and even may destroy the intrinsic structure of the observed data.

In this paper, we propose a cyclic coordinate descent (CCD) algorithm (called CCD algorithm) for solving $l_q$ regularization problem without the requirement of column-normalization.
Instead, we introduce a stepsize parameter to improve the applicability of the CCD algorithm.
The proposed CCD algorithm can be viewed as a variant of the $l_q$CD algorithm.
In the perspective of algorithmic implementation,
it can be noted that the $l_q$CD algorithm proposed in {\cite{lqCD-Marjanovic-2014}} is actually a special case of the proposed CCD algorithm with the stepsize as 1 and a column-normalized $A$.
More important, we can justify the convergence instead of the subsequential convergence of the proposed CCD algorithm via introducing a stepsize parameter.
We prove that the proposed CCD algorithm can converge to a stationary point
as long as the stepsize less than $\frac{1}{L_{\max}}$ with $L_{\max} = \max_i \|A_i\|_2^2$.
This convergence condition is generally weaker than those of the iterative thresholding algorithms, i.e., the stepsize parameter should be less than $\|A\|_2^{-2}$ {\cite{ZengIJT2014}}, {\cite{ZengHalf2014}}.
Roughly, the proposed CCD algorithm is a Gauss-Seidel iterative method while the corresponding iterative thresholding algorithm is a Jacobi iterative algorithm.
We can also justify that the proposed CCD algorithm converges to a local minimizer under some additional conditions.
In addition, it can be observed numerically that the proposed CCD algorithm has almost the same performance of the $l_q$CD algorithm when $A$ is normalized in column and the stepsize approaches to 1.

The reminder of this paper is organized as follows.
In section II, we first introduce the $l_q$ $(0<q<1)$ regularization problem,
then propose a cyclic coordinate descent algorithm for such a non-convex  regularization problem.
In section III, we prove the convergence of the proposed CCD algorithm.
%In section IV, we further show that the proposed algorithm converges to a local minimizer under certain conditions.
In section IV, we implement a series of simulations to demonstrate the efficiency of the CCD algorithm.
We conclude this paper in section V.
%, and present some proofs in the Appendix.

{\bf Notations:} We denote $\mathbf{N}$ and $\mathbf{R}$ as the natural number set and one-dimensional real space, respectively.
Given an index set $I \subset \{1,2,\cdots, N\}$, $I^c$ represents its complementary set,
i.e., $I^c = \{1,2,\cdots, N\} \setminus I.$
For any matrix $A\in \mathbf{R}^{M\times N}$, $A_i$ denotes as the $i$-th column of $A$,
and $A_I$ represents a submatrix of $A$ with the columns restricted to an index set $I$.
Similarly, for any vector $x\in \mathbf{R}^N$, $x_i$ denotes as the $i$-th component of $x$,
and $x_I$ represents a subvector of $x$ with the coordinate coefficients restricted to $I$.
For any matrix and vector, we denote $\cdot^T$ by the transpose operation.
For any square matrix $A$, $\sigma_i(A)$ and $\sigma_{\min}(A)$ denote as the $i$-th and the minimal eigenvalues of $A$, respectively.

\section{A Cyclic Coordinate Descent Algorithm}

In this section, we first introduce the non-convex $l_q$ regularization ($0<q<1$) problem,
then show some important theoretical results of the $l_q$ regularization problem, which serve as the basis of the following sections.
Finally, we propose a cyclic coordinate descent (CCD) algorithm for solving the $l_q$ regularization problem.

\subsection{$l_q$ Regularization Problem}

Mathematically, the $l_q$ regularization problem is
\begin{equation}
\min_{x\in \mathbf{R}^N} \left\{T_{\lambda}(x) = \frac{1}{2}{\Vert Ax-y\Vert}_{2}^{2} + \lambda \|x\|_q^q \right\},
\label{lqReg}
\end{equation}
where $0<q<1$ and $\lambda>0$.
It can be easily observed that the first least squares term is proper lower semi-continuous while the $l_q$ penalty is continuous and coercive,
and thus the minimum of the $l_q$ regularization problem exists.
However, due to the non-convexity, the $l_q$ regularization problem might have several global minimizers.

For better characterizing the global minimizers of (\ref{lqReg}), we first generalize the proximity operator from convex case to the non-convex $l_q$ norm,
\begin{equation}
Prox_{\lambda\mu,q}(x) = \arg \min_{u\in \mathbf{R}^N} \left\{\frac{\|x-u\|_2^2}{2} + \lambda\mu \|u\|_q^q\right\},
\label{ProxOper}
\end{equation}
where $\mu>0$ is the stepsize parameter.
Since $\|\cdot\|_q^q$ is separable, thus computing $Prox_{\lambda\mu,q}$ is reduced to solve a one-dimensional minimization problem, that is,
\begin{equation}
prox_{\lambda\mu,q}(z) = \arg \min_{v\in \mathbf{R}} \left\{\frac{|z-v|^2}{2} + \lambda\mu |v|^q\right\},
\label{SVProxOper}
\end{equation}
and thus,
\begin{equation}
Prox_{\lambda\mu,q}(x) = (prox_{\lambda\mu,q}(x_1), \cdots, prox_{\lambda\mu,q}(x_N))^T.
\end{equation}
Furthermore, according to {\cite{BrediesNonconvex}}, $prox_{\lambda\mu, q}(\cdot)$ can be expressed as follows: \\
$prox_{\lambda\mu, q} (z)=$
\begin{equation}
\left\{
\begin{array}
[c]{ll}%
(\cdot + \lambda \mu q sgn(\cdot) |\cdot|^{q-1})^{-1}(z), & \mbox{for} \ |z|\geq \tau_{\mu,q}\\
0, & \mbox{for} \ |z|\leq \tau_{\mu,q}
\end{array}
\right.
\label{ProxMapExpLq}
\end{equation}
for any $z \in \mathbf{R}$
with
\begin{equation}
\tau_{\mu,q} = \frac{2-q}{2-2q}(2\lambda\mu(1-q))^{\frac{1}{2-q}},
\label{ThreshValuexLq}
\end{equation}
\begin{equation}
\eta_{\mu,q} = (2\lambda\mu (1-q))^{\frac{1}{2-q}},
\label{ThreshValueyLq}
\end{equation}
and the range domain of $prox_{\lambda\mu,q}$ is $\{0\}\cup [\eta_{\mu,q},\infty)$, $sgn(\cdot)$ represents the sign function henceforth.
When $\ |z|\geq \tau_{\mu,q}$, the relation $prox_{\lambda\mu, q} (z)= (\cdot + \lambda \mu q sgn(\cdot) |\cdot|^{q-1})^{-1}(z)$ means that
$prox_{\lambda\mu, q} (z)$ satisfies the following equation
\[
v + \lambda \mu q\cdot sgn(v) |v|^{q-1} = z.
\]

\begin{remark} \label{Remark_eg_ThresholdingFun}
From (\ref{ProxMapExpLq}), it can be noted that $prox_{\lambda\mu,q}$ is a set-valued operator since it can take two different function values when $|z|=\tau_{\mu,q}.$
Moreover, for some specific $q$ (say, $q=1/2, 2/3$), the operator $prox_{\lambda\mu,q}(\cdot)$ can be expressed analytically, which are shown as follows:
%Some analytical examples of $prox_{\lambda\mu,q}$ are shown as follows:
%\\
%{\bf Example 1}. Hard thresholding function for $l_{0}$ regularization
%\begin{equation}
%prox_{\lambda\mu,0}(z)= \left\{
%\begin{array}{cc}
%z,  & |z|\geq \tau_{\mu,0} \\
%0, & |z|\leq \tau_{\mu,0}
%\end{array}%
%\right..
%\label{HardThreshFun}
%\end{equation}%

(a) $prox_{\lambda\mu,1/2}(\cdot)$ for $l_{1/2}$ regularization (\cite{L1/2TNN}):
\\
$prox_{\lambda\mu,1/2}(z)=$
\begin{equation}
\\
\left\{
\begin{array}{cc}
{\frac{2}{3}}z
\left(1 + \cos\left({\frac{2{\pi}}{3}}-{\frac{2}{3}}\theta_{\tau_{\mu,1/2}}(z)\right)\right),  &
|z|\geq \tau_{\mu,1/2} \\
0, & |z|\leq \tau_{\mu,1/2}%
\end{array}%
\right.,
\label{HalfThreshFun}
\end{equation}%
with $\theta_{\tau_{\mu,1/2}}(z) = \arccos\left({\frac{\sqrt{2}}{2}}{(\frac{\tau_{\mu,1/2}}{|z|})}^{\frac{3}{2}}\right).$

(b) $prox_{\lambda\mu,2/3}(\cdot)$ for $l_{2/3}$ regularization ({\cite{L2/3Cao2013}}):
\\
$prox_{\lambda\mu,2/3}(z)=$
\begin{equation}
\left\{
\begin{array}{cc}
sgn(z) \left( \frac{\phi_{\tau_{\mu,2/3}}(z)+\sqrt{f_{\tau_{\mu,2/3}}(z)}}{2}\right)^3,  &
|z|\geq \tau_{\mu,2/3} \\
0, & |z|\leq \tau_{\mu,2/3}%
\end{array}%
\right., \label{2/3ThreshFun}
\end{equation}%
where
$$f_{\tau_{\mu,2/3}}(z) = \frac{2|z|}{\phi_{\tau_{\mu,2/3}}(z)}-\phi_{\tau_{\mu,2/3}}(z)^2,$$
and
$$\phi_{\tau_{\mu,2/3}}(z) = \frac{2^{13/16}}{4\sqrt{3}}\tau_{\mu,2/3}^{3/16}(\cosh(\frac{\theta_{\tau_{\mu,2/3}}(z)}{3}))^{1/2}$$
with $\theta_{\tau_{\mu,2/3}}(z) = arccosh(\frac{3\sqrt{3}z^2}{2^{7/4}(2\tau_{\mu,2/3})^{9/8}})$.
\end{remark}

\begin{remark} \label{Remark_ApproxThreshFun}
It was demonstrated in {\cite{L1/2TNN}} that $prox_{\lambda\mu,q}(\cdot)$ has analytical expression when $q$ is $\frac{1}{2}$ or $\frac{2}{3}$.
While for other $q\in (0,1)$, we can use an iterative scheme proposed by {\cite{lqCD-Marjanovic-2014}} to compute the operator $prox_{\lambda\mu,q}(\cdot)$, that is, let $\eta_{\mu,q} \leq v^0 \leq |z|$,
\begin{equation*}
v^{k+1} = |z|-\lambda\mu q |v^k|^{q-1}.
\end{equation*}
\end{remark}

With the definition of proximity operator, we can define a new operator $G_{\lambda\mu,q}(\cdot)$ as
\begin{equation}
G_{\lambda\mu,q}(x) = Prox_{\lambda\mu,q}(x-\mu A^T(Ax-y))
\label{G_oper}
\end{equation}
for any $x\in \mathbf{R}^N$.
We denote ${\mathcal F}_q$ as the fixed point set of the operator $G_{\lambda\mu,q}$, i.e.,
\begin{equation}
{\mathcal F}_q = \{x: x = G_{\lambda\mu,q}(x)\}.
\label{FixedPointSet}
\end{equation}

%{\bf Lemma 2 (Proposition 2.3 in {\cite{BrediesNonconvex}}).} Assume that $0<\mu<\|A\|_2^{-2}$, then each global minimizer of $T_{\lambda}$ is a fixed point of $G_{\lambda\mu,q}(\cdot)$.
According to {\cite{BrediesNonconvex}}, each global minimizer of (\ref{lqReg}) is
a fixed point of $G_{\lambda\mu,q}(\cdot)$, which is shown as follows.
\begin{lemma} \label{Lemma_Minimization-FixedPoint}
(Proposition 2.3 in {\cite{BrediesNonconvex}}).
Assume that $0<\mu<\|A\|_2^{-2}$, then each global minimizer of $T_{\lambda}$ is a fixed point of $G_{\lambda\mu,q}(\cdot)$.
\end{lemma}

By the definition of $Prox_{\lambda\mu,q}$, a type of optimality conditions of $l_q$ regularization has been derived in {\cite{lqCD-Marjanovic-2014}}.

\begin{lemma} \label{Lemma_OptimalCondition}
(Theorem 3 in {\cite{lqCD-Marjanovic-2014}}). Given a point $x^*$, define the support set of $x^*$ as $Supp(x^*) = \{i: x_i^* \neq 0\}$, then $x^* \in {\mathcal F}_q$ if and only if the following three conditions hold.
\begin{enumerate}
\item[(a)]
For $i\in Supp(x^*)$, $|x_i^*| \geq \eta_{\mu,q}$.

\item[(b)]
For $i\in Supp(x^*)$, $A_i^T(Ax^*-y) + \lambda q sgn(x_i^*) |x_i^*|^{q-1} = 0 $.

\item[(c)]
For $i\in Supp(x^*)^c$, $|A_i^T(Ax^*-y)| \leq \tau_{\mu,q}/\mu$.
\end{enumerate}
\end{lemma}

We call the point a {\bf stationary point} of the $l_q$ regularization problem henceforth if it satisfies the optimality conditions in Lemma {\ref{Lemma_OptimalCondition}}.

\subsection{A CCD Algorithm for $l_q$ Regularization}

In this subsection,
we derive a cyclic coordinate descent algorithm for solving the $l_q$ regularization problem.
%This algorithm is similar to the $l_q$CD algorithm studied in {\cite{lqCD-Marjanovic-2014}}.
More specifically, given the current iterate $x^n$, at the next iteration, the $i$-th coefficient is selected by
\begin{equation}
  i= \left\{
  \begin{array}{cc}
  N   & {\rm if}\  0\equiv {(n+1)}\  {\rm mod} \ N \\
  {(n+1)}\  {\rm mod} \ N, & {\rm otherwise} %
  \end{array}%
  \right.
  \label{updatingindex}
\end{equation}%
and then updated by
\[
x_i^{n+1} \in \arg\min_{v\in \mathbf{R}} \left\{\frac{|z^n_i-v|^2}{2} + \lambda\mu |v|^q\right\} = prox_{\lambda\mu,q}(z^n_i),
\]
where
\begin{equation}
z_i^n = x_i^n - \mu A_i^T(Ax^n-y).
\label{updateFowStep}
\end{equation}
It can be seen from (\ref{ProxMapExpLq}) that $prox_{\lambda\mu,q}$ is a set-valued operator.
Therefore, we select a particular single-valued operator of $prox_{\lambda\mu,q}$ and then update $x_i^{n+1}$ according to the following scheme,
\begin{equation}
x^{n+1}_i= \mathcal{T} (z_i^n, x_i^n),
\label{updatingrul1}
\end{equation}
where
\begin{equation*}
\mathcal{T} (z_i^n, x_i^n)=
  \left\{
  \begin{array}{cc}
  prox_{\lambda\mu,q}(z_i^{n})   & {\rm if} \ |z_i^{n}| \neq \tau_{\mu,q} \\
  sgn(z_i^{n})\eta_{\mu,q} \mathbf{I}(x_i^{n} \neq 0), & {\rm if} \ |z_i^{n}| = \tau_{\mu,q}%
  \end{array}%
  \right.,
%  \label{updatingrul1}
\end{equation*}%
and $\mathbf{I}(x_i^{n} \neq 0)$ denotes the indicator function, that is,
\begin{equation*}
  \mathbf{I}(x_i^{n} \neq 0)= \left\{
  \begin{array}{cc}
  1,   & {\rm if}\  x_i^{n} \neq 0\\
  0, & {\rm otherwise} %
  \end{array}%
  \right..
  %\label{updatingindex}
\end{equation*}%
While the other components of $x^{n+1}$ are being fixed, i.e.,
\begin{equation}
x_j^{n+1} = x_j^{n}, \ {\rm for}\ j\neq i.
\label{updatingrul2}
\end{equation}

In summary, we can formulate the proposed algorithm as follows.

\begin{center}
\text{The Cyclic Coordinate Descent (CCD) Algorithm}\\ \vspace{0.1cm}
\begin{tabular}{l}
  \hline
  %\vspace{0.1cm}
  Initialize with $x^0$. Choose a stepsize $\mu>0$, let $n:=0$. \\ \vspace{0.2cm}
  Step 1. Calculate the index $i$ according to (\ref{updatingindex}); \\ \vspace{0.2cm}
  Step 2. Calculate $z_i^n$ according to (\ref{updateFowStep}); \\ \vspace{0.2cm}
  Step 3. Update $x_i^{n+1}$ via (\ref{updatingrul1}) and $x_j^{n+1} = x_j^{n}$  for $j\neq i$;\\ \vspace{0.2cm}
  Step 4. Check the terminational rule. If yes, stop; \\ \vspace{0.2cm}
  \ \ \ \ \ \ \ \ \ otherwise, let $n:= n+1$, go to Step 1.\\
%  Step 5.  \\
  \hline
\end{tabular}\\
\end{center}

%{\bf Remark 2.}
\begin{remark}\label{Remark_Comp_lqCD}
It can be observed that the proposed algorithm is similar to the $l_q$CD proposed by Marjanovic and Solo {\cite{lqCD-Marjanovic-2014}}.
However, we get rid of the column-normalization requirement of $A$ by introducing a stepsize parameter $\mu$. The following sections show that it can bring more benefits in both algorithmic implementation and theoretical justification.
%It improves the flexibility of the algorithm
%The proposed algorithm is similar to the $l_q$ cyclic descent algorithm ($l_q$CD) proposed by Marjanovic and Solo {\cite{lqCD-Marjanovic-2014}}.
%It can be observed that the $l_q$CD algorithm is a special case of the proposed CCD algorithm in this paper with the step-size as 1
%and the columns of $A$ being normalized.
%However, this generalization is non-trivial.
%In the next sections, we will show that it will bring many benefits in the theoretical justification as well as the numerical performance via introducing such a step-size parameter.
\end{remark}

\section{Convergence Analysis}

In this section, we prove the convergence of the proposed CCD algorithm for the $l_q$ regularization problem with $0<q<1$.
We first give some basic properties of the proposed algorithm, which serve as the basis of the next subsections,
and then prove that the CCD algorithm converges to a stationary point from any initial point as long as the stepsize parameter $\mu$ is less than a positive constant,
and finally show that the proposed algorithm converges to a local minimizer under certain additional conditions.

\subsection{Some Basic Properties of CCD Algorithm}

According to the definition of the operator $prox_{\lambda\mu,q}(\cdot)$ (\ref{ProxMapExpLq}) and the updating rule of CCD algorithm (\ref{updatingindex})-(\ref{updatingrul2}),
we can claim that $x_i^{n+1}$ satisfies the following property.

\begin{property}\label{Prop_OptimalCondition_n+1}
Given the current iterate $x^n$ ($n \in \mathbf{N}$), the index set $i$ is determined via (\ref{updatingindex}), then $x_i^{n+1}$ satisfies either
\begin{enumerate}
\item[(a)]
$x_i^{n+1} = 0,$
or,
\item[(b)]
$|x_i^{n+1}| \geq \eta_{\mu,q}$ and also satisfies the following equation
\begin{align}
&A_i^T(Ax^{n+1}-y) + \lambda q sgn(x_i^{n+1}) |x_i^{n+1}|^{q-1} \nonumber\\
&= (\frac{1}{\mu}-A_i^TA_i)(x_i^n - x_i^{n+1}).
\label{OptCon_i}
\end{align}
that is, $\nabla_i T_{\lambda}(x^{n+1}) = (\frac{1}{\mu}-A_i^TA_i)(x_i^n - x_i^{n+1}),$
where $\nabla_i T_{\lambda}(x^{n+1})$ represents the gradient of $T_{\lambda}$ with respect to the $i$-th coordinate at the point $x^{n+1}$.
\end{enumerate}
\end{property}

\begin{proof}
According to (\ref{ProxMapExpLq}) and (\ref{updatingrul1}), it holds obviously either $x_i^{n+1} =0$ or $|x_i^{n+1}| \geq \eta_{\mu,q}$.
Moreover, when $|x_i^{n+1}| \geq \eta_{\mu,q}$, according to (\ref{SVProxOper}),
$x_i^{n+1}$ is a minimizer of the optimization problem (\ref{SVProxOper}) with $z=z_i^n$.
Therefore, $x_i^{n+1}$ should satisfy the following optimality condition
\begin{equation}
x_i^{n+1} - z_i^n  + \lambda \mu q sgn(x_i^{n+1}) |x_i^{n+1}|^{q-1} = 0.
\label{OptCon1}
\end{equation}
Plugging (\ref{updateFowStep}) into (\ref{OptCon1}) gives
\begin{align}
&A_i^T(Ax^{n+1}-y) + \lambda q sgn(x_i^{n+1}) |x_i^{n+1}|^{q-1} \nonumber\\
&= \frac{1}{\mu}(x_i^n - x_i^{n+1})-A_i^TA(x^n-x^{n+1}).
\label{OptCon2}
\end{align}
Combining (\ref{updatingrul2}) and (\ref{OptCon2}) implies (\ref{OptCon_i}).
\end{proof}

As shown by Property 1, the coordinate-wise gradient of $T_{\lambda}$ with respect to the $i$-th coordinate at $x^{n+1}$
is not exact zero but with a relative error.
In the following, we show that the sequence $\{T_{\lambda}(x^n)\}$ satisfies the sufficient decrease property {\cite{Burke1985SuffCon}}.

\begin{property}\label{Prop_SufficientDecrease}
Let $\{x^n\}$ be a sequence generated by the CCD algorithm. Assume that $0<\mu<L_{\max}^{-1}$, then
\[
T_{\lambda}(x^n) - T_{\lambda}(x^{n+1}) \geq \frac{1}{2}(\frac{1}{\mu} - L_{\max}) \|x^n - x^{n+1}\|_2^2, ~\forall n \in \mathbf{N}.
\]
\end{property}

\begin{proof}
Given the current iteration $x^n$, let the coefficient index $i$ be determined according to (\ref{updatingindex}).
According to (\ref{SVProxOper}) and (\ref{updatingrul1}),
\begin{equation*}
x_i^{n+1} \in \arg \min_{v\in \mathbf{R}} \left\{\frac{|z_i^n-v|^2}{2} + \lambda\mu |v|^q\right\},
\end{equation*}
where $z_i^n = x_i^n -\mu A_i^T (Ax^n-y)$.
Then it implies
\begin{align*}
& \frac{1}{2}|\mu A_i^T(Ax^n-y)|^2 + \lambda\mu |x_i^n|^q \nonumber\\
& \geq \frac{1}{2} |(x_i^{n+1} - x_i^n) + \mu A_i^T (Ax^n -y)|^2 + \lambda \mu |x_i^{n+1}|^q.
\end{align*}
Some simplifications give
\begin{align}
&\lambda |x_i^n|^q - \lambda |x_i^{n+1}|^q  \nonumber\\
&\geq \frac{|x_i^{n+1} - x_i^n|^2}{2\mu} + A_i^T(Ax^n-y)(x_i^{n+1} - x_i^n).
\label{Property2.1}
\end{align}
Moreover, since $x_j^{n+1} = x_j^n$ for any $j\neq i$,
%by (\ref{updatingrul2}),
(\ref{Property2.1}) becomes
\begin{align}
&\lambda \|x^n\|_q^q - \lambda \|x^{n+1}\|_q^q  \nonumber\\
&\geq \frac{\|x^{n+1} - x^n\|^2}{2\mu} + \langle Ax^n-y, A(x^{n+1} - x^n) \rangle.
\label{Property2.2}
\end{align}
Adding $\frac{1}{2} \|Ax^n -y\|_2^2 - \frac{1}{2} \|Ax^{n+1} -y\|_2^2$ to both sides of (\ref{Property2.2}) gives
\begin{align}
& T_{\lambda}(x^n) - T_{\lambda}(x^{n+1}) \nonumber\\
& \geq \frac{\|x^{n+1} - x^n\|^2}{2\mu} - \frac{1}{2}\|A(x^n - x^{n+1})\|_2^2 \nonumber\\
& = \frac{\|x^{n+1} - x^n\|^2}{2\mu} - \frac{1}{2}(A_i^TA_i)\|x^n - x^{n+1}\|_2^2 \nonumber\\
& \geq \frac{1}{2}(\frac{1}{\mu} - L_{\max}) \|x^n - x^{n+1}\|_2^2,
\label{Property2.3}
\end{align}
where the first equality holds for
$$
\|A(x^n - x^{n+1})\|_2^2 = (A_i^TA_i)|x_i^n - x_i^{n+1}|^2 =(A_i^TA_i)\|x^n - x^{n+1}\|_2^2,
$$
and the second inequality holds for $A_i^TA_i \leq L_{\max}$.
\end{proof}

In fact, by the first inequality of (\ref{Property2.3}), a slightly stricter but more commonly used condition to guarantee the sufficient decrease is $0<\mu<\|A\|_2^{-2}.$ Property {\ref{Prop_SufficientDecrease}}, gives the boundedness of the sequence $\{x^n\}$.

\begin{property}\label{Prop_Boundedness}
Let $\{x^n\}$ be a sequence generated by the CCD algorithm.
Assume that $0<\mu<L_{\max}^{-1}$ and $T_{\lambda}(x^0)<+\infty$, then $x^n$ is bounded for any $n\in \mathbf{N}$.
\end{property}

\begin{proof}
By Property {\ref{Prop_SufficientDecrease}}, for any $n$,
\begin{align*}
\|x^n\|_q^q \leq \frac{1}{\lambda} T_{\lambda}(x^n) \leq \frac{1}{\lambda} T_{\lambda}(x^0) < +\infty.
\end{align*}
It implies that $x^n$ is bounded.
\end{proof}

Moreover, Property {\ref{Prop_SufficientDecrease}} also gives the following asymptotically regular property.

\begin{property}\label{Prop_AsymptoticRegular}
Let $\{x^n\}$ be a sequence generated by the CCD algorithm.
Assume $0<\mu<L_{\max}^{-1}$, then
$$
\sum_{k=0}^n \|x^{k+1} - x^{k}\|_2^2 \leq \frac{2\mu}{1-\mu L_{\max}} T_{\lambda}(x^0),
$$
and
$$
\|x^n - x^{n+1}\|_2 \rightarrow 0, \ {\rm as}\ n\rightarrow +\infty.
$$
\end{property}

From Properties {\ref{Prop_SufficientDecrease}}-{\ref{Prop_AsymptoticRegular}}, we can prove the subsequential convergence of the CCD algorithm.

\begin{theorem}\label{Thm_SubsequentialConv}
Let $\{x^n\}$ be a sequence generated by the CCD algorithm.
Assume that $0<\mu<L_{\max}^{-1}$ and $T_{\lambda}(x^0) <+\infty$,
then the sequence $\{x^n\}$ has a convergent subsequence.
Moreover, let $\mathcal{L}$ be the set of the limit points of $\{x^n\}$,
then $\mathcal{L}$ is closed and connected.
\end{theorem}

\begin{proof}
By Property {\ref{Prop_SufficientDecrease}}, we know that $\{T_{\lambda}(x^n)\}$ is a decreasing and lower-bounded sequence,
thus, $\{T_{\lambda}(x^n)\}$ is convergent. Denote the convergent value of $\{T_{\lambda}(x^n)\}$ as $T^*$.
Moreover, by Property {\ref{Prop_Boundedness}}, $\{x^n\}$ is bounded, and also by the continuity of $T_{\lambda}(\cdot)$,
there exists a subsequence of $\{x^n\}$, $\{x^{n_j}\}$ converging to some point $x^*$, which satisfies $T_{\lambda}(x^*) = T^*$.

Furthermore, by Property {\ref{Prop_AsymptoticRegular}} and Ostrowski's result (Theorem 26.1, p. 173) {\cite{Ostrowski1973}},
the limit point set $\mathcal{L}$ of the sequence $\{x^n\}$ is closed and connected.
\end{proof}

Theorem 1 only shows the subsequential convergence of the CCD algorithm.
Moreover, we note that $\mathcal{L}$ might not be a set of isolated points.
Due to this, it becomes challenging to justify the global convergence of CCD algorithm.
More specifically, there are still two open questions on the convergence of the CCD algorithm.
\begin{enumerate}
\item[(a)]
When does the algorithm converge globally? So far, for most non-convex algorithms, only subsequential convergence can be claimed.

\item[(b)]
Where does the algorithm converge? Does the algorithm converge to a global minimizer or more practically,
a local minimizer due to the non-convexity of the optimization problem?
\end{enumerate}

\subsection{Convergence To A Stationary Point}

In this subsection, we will focus on answering the first open question proposed in the end of the last subsection.
More specifically, we will show that the whole sequence $\{x^n\}$ generated by the CCD algorithm converges to a stationary point as long as the stepsize parameter $\mu$ satisfies $0<\mu<\frac{1}{L_{\max}}$.
%, and further converges to a stationary point of $l_q$ regularization if $\mu \in (0,\|A\|_2^{-2}).$

%In order to prove the global convergence of CCD algorithm, we first show some proerties of the algorithm mapping $\mathcal{T}(\cdot,\cdot)$.
Given the current iteration $x^n$, we define the descent function as
\begin{equation}
\Delta(x^n, x^{n+1}) = T_{\lambda}(x^n) - T_{\lambda}(x^{n+1}).
\label{DescentFun}
\end{equation}
Note that $x^n$ and $x^{n+1}$ differ only in their $i$-th coefficient which is determined by (\ref{updatingindex}).
%Moreover, from (\ref{updatingrul1}), it is clear that
%\begin{equation}
%x_i^{n+1} = \mathcal{T}(z_i^n, x_i^n).
%\label{}
%\end{equation}
From now on, if not stated, it is assumed $x_i^{n+1}$ is given by (\ref{updatingrul1}) and $i$ is given by (\ref{updatingindex}).
The following lemma gives an important property of the descent function.

\begin{lemma}\label{Lemma_DescentFun}
Let $\{x^n\}$ be a sequence generated by CCD algorithm.
Assume that $0<\mu<L_{\max}^{-1}$, then
\[
\Delta(x^n, x^{n+1}) = 0 \ {\rm if~ and~ only~ if} \ x_i^{n+1} = x_i^n.
\]
\end{lemma}
\begin{proof}
($\Rightarrow$) It is obvious that if $x_i^{n+1} = x_i^n$, then $x^{n+1} = x^n$, and thus $\Delta(x^n, x^{n+1}) = 0$.

($\Leftarrow$) If $\Delta(x^n, x^{n+1}) = 0$, then Property {\ref{Prop_SufficientDecrease}} implies $x^{n+1} = x^n$ and thus, $x_i^{n+1} = x_i^n$.
\end{proof}

Moreover, similar to Theorem 10 in {\cite{lqCD-Marjanovic-2014}},
we can claim that the mapping $\mathcal{T}(\cdot,\cdot)$ is a closed mapping, shown as follows.

%{\bf Lemma 4 (Theorem 10 in {\cite{lqCD-Marjanovic-2014}}).}
\begin{lemma}\label{Lemma_ClosedMap}
$\mathcal{T}(\cdot,\cdot)$ is a closed mapping, i.e., assume
\begin{enumerate}
\item[(a)]
$x_i^n \rightarrow x_i^*$ as $n \rightarrow \infty;$

\item[(b)]
$x_i^{n+1} \rightarrow x_i^{**}$ as $n\rightarrow \infty$, where $x_i^{n+1} = \mathcal{T}(z_i^n, x_i^n).$
\end{enumerate}
Then $x_i^{**} = \mathcal{T}(z_i^*, x_i^*)$, where $z_i^* = x_i^* - \mu A_i^T(Ax^*-y)$.
\end{lemma}

The proof is the essentially the same as that of Theorem 10 in {\cite{lqCD-Marjanovic-2014}}. The only difference is that  $prox_{\lambda\mu,q}$ is discontinuous at $\tau_{\mu,q}$ while $prox_{\lambda,q}$ is discontinuous at $\tau_{1,q}$.
Therefore, the closedness of the operator $\mathcal{T}(\cdot,\cdot)$ can not be changed after introducing a stepsize $\mu$.
The following theorem shows that any limit point of the sequence $\{x^n\}$ is a stationary point of the non-convex $l_q$ regularization problem.

\begin{theorem}\label{Thm_LimitPoint_StatPoint}
Let $\{x^n\}$ be a sequence generated by the CCD algorithm, and $\mathcal{L}$ be its limit point set.
Assume that $0<\mu<L_{\max}^{-1}$ and $T_{\lambda}(x^0)<+\infty$, then $\mathcal{L} \subseteq {\mathcal{F}}_q$.
\end{theorem}

The proof of this theorem is similar to that of Theorem 5 in {\cite{lqCD-Marjanovic-2014}}.
For the completion, we provide the proof as follows.

\begin{proof}
Since the sequence $\{x^n\}$ is bounded, then it has limit points.
Let $x^* \in \mathcal{L}$. We now focus on the $i$-th coefficient of the sequence with $n=n(i)=jN+i-1$, where $i=1,2,\dots, N$ and $j=0,1,\dots.$
However, here, we simply use $n$ by which we mean $n(i)$.
Now there exists a subsequence $\{x^{n_1}, x^{n_2}, \cdots\}$ such that
\begin{equation}
\{x^{n_1}, x^{n_2}, \cdots\} \rightarrow x^* \ {\rm and} \ \{x_i^{n_1}, x_i^{n_2}, \cdots\} \rightarrow x_i^*.
\label{subseq1}
\end{equation}
Moreover, since the sequence $\{x^{n_1+1}, x^{n_2+1}, \cdots\}$ is also bounded, thus, it also has limit points.
Denoting one of these by $x^{**}$, then there exists a subsequence $\{x^{l_1+1}, x^{l_2+1}, \cdots\}$ such that
\begin{equation}
\{x^{l_1+1}, x^{l_2+1}, \cdots\} \rightarrow x^{**} \ {\rm and} \ \{x_i^{l_1+1}, x_i^{l_2+1}, \cdots\} \rightarrow x_i^{**},
\label{subseq2}
\end{equation}
where $\{l_1,l_2,\cdots\} \subset \{n_1, n_2, \cdots\}$.
In this case, it holds
\begin{equation}
\{x^{l_1}, x^{l_2}, \cdots\} \rightarrow x^* \ {\rm and} \ \{x_i^{l_1}, x_i^{l_2}, \cdots\} \rightarrow x_i^*,
\label{subseq3}
\end{equation}
since it is a subsequence of (\ref{subseq1}).
From (\ref{updateFowStep}) and (\ref{subseq3}), we have
\begin{equation*}
z_i^{l_j} \rightarrow z_i^* \ {\rm as}\ j \rightarrow \infty.
\end{equation*}
Thus, by Lemma {\ref{Lemma_ClosedMap}}, it holds
\begin{equation}
x_i^{**} = \mathcal{T}(z_i^*,x_i^*).
\label{Th2.1}
\end{equation}
Moreover, by (\ref{subseq2}), (\ref{subseq3}) and (\ref{updatingrul2}), it holds
\begin{equation}
x_j^* = x_j^{**} \ {\rm for} \ j\neq i.
\label{Th2.2}
\end{equation}

In the following,
by the continuity of $T_{\lambda}(\cdot)$ and thus the continuity of $\Delta(\cdot,\cdot)$ with respect to its arguments,
it holds
\[
\Delta(x^{l_j}, x^{l_j+1}) \rightarrow \Delta(x^*,x^{**}).
\]
Moreover, since the sequence $\{T_{\lambda}(x^n)\}$ is convergent, then
\[
\Delta(x^{l_j}, x^{l_j+1}) = T_{\lambda}(x^{l_j}) - T_{\lambda}(x^{l_j+1}) \rightarrow 0 \ {\rm as}\ j\rightarrow \infty,
\]
which implies
\[
\Delta(x^*,x^{**}) = 0.
\]
Furthermore, by Lemma {\ref{Lemma_DescentFun}}, and (\ref{Th2.1})-(\ref{Th2.2}), it holds
\begin{equation}
x_i^{**} = x_i^*.
\label{Th2.3}
\end{equation}
Combining (\ref{Th2.1}) and (\ref{Th2.3}), we have
\begin{equation}
x_i^* = \mathcal{T}(z_i^*,x_i^*).
\label{Th2.4}
\end{equation}

Since $i$ is arbitrary, we have that (\ref{Th2.4}) holds for all $i \in \{1,\cdots, N\}$.
It implies that $x^*$ is a fixed point of $G_{\lambda\mu,q}(\cdot)$, that is, $x^* \in {\mathcal{F}}_q$.
Similarly, since $x^* \in \mathcal{L}$ is also arbitrary, therefore, $\mathcal{L} \subset {\mathcal{F}}_q$.
Consequently, we complete the proof of this theorem.
\end{proof}

In the following theorem, we demonstrate the finite support convergence of the sequence $\{x^n\}$,
that is, the support of $\{x^n\}$ will converge within a finite number of iterations.
Denote $I^n = Supp(x^n) = \{i: |x_i^n| \neq 0, i=1,2\cdots, N\}$, $I = Supp(x^*)$.

\begin{theorem}\label{Thm_SupportConv}
Let $\{x^n\}$ be a sequence generated by the CCD algorithm.
Assume that $0<\mu<\frac{1}{L_{\max}}$ and $x^*$ is a limit point of $\{x^n\}$,
then there exists a sufficiently large positive integer $n^*>N$ such that when $n > n^*$, it holds
\begin{enumerate}
\item[(a)]
either $x_j^n=0$ or $|x_j^n|\geq \eta_{\mu,q}$ for $j=1,2,\cdots,N;$

\item[(b)]
$I^n=I$;

\item[(c)]
$sgn(x^n) = sgn(x^*)$.
\end{enumerate}
\end{theorem}

\begin{proof}
We can note that all the coefficient indices will be updated at least one time when $n>N$.
By Property {\ref{Prop_OptimalCondition_n+1}}, once the index $i$ is updated at $n$-th iteration, then the coefficient $x_i^n$ satisfies:
\[
{\rm either}\  x_i^n=0 \ {\rm or}\  |x_i^n|\geq \eta_{\mu,q}.
\]
Thus, Theorem {\ref{Thm_SupportConv}}(a) holds.

In the following, we prove Theorem {\ref{Thm_SupportConv}}(b) and (c).
By the assumption of Theorem {\ref{Thm_SupportConv}}, there exits a subsequence $\{x^{n_j}\}$ converges to $x^*$, i.e.,
\begin{equation}
x^{n_j} \rightarrow x^*\ \ \text{as}\ \ j\rightarrow \infty.
\end{equation}
Thus, there exists a sufficiently large positive integer $j_0$ such that $\|x^{n_j} - x^*\|_2 < \eta_{\mu,q}$ when $j\geq j_0$.
Moreover, by Property {\ref{Prop_AsymptoticRegular}}, there also exists a sufficiently large positive integer $n^*>N$ such that $\|x^{n} - x^{n+1}\|_2 < \eta_{\mu,q}$ when $n>n^*$.
Without loss of generality, we let $n^* = n_{j_0}$.
In the following, we first prove that $I^n = I$ and $sgn(x^n) = sgn(x^*)$ whenever $n>n^*$.

In order to prove $I^n=I$, we first show that $I^{n_j} = I$ when $j\geq j_0$ and then verify that $I^{n+1} = I^n$ when $n>n^*$.
We now prove by contradiction that $I^{n_j} = I$ whenever $j\geq j_0$.
Assume this is not the case, namely, that $I^{n_j} \neq I$.
Then we easily derive a
contradiction through distinguishing the following two possible cases:

\textit{Case 1:} $I^{n_j}\neq I$ and $I^{n_j}\cap I\subset I^{n_j}.$ In this case,
then there exists an $i_{n_j}$ such that $i_{n_j}\in I^{n_j}\setminus I$.
By Theorem {\ref{Thm_SupportConv}}(a), it then implies
\[
\Vert x^{n_j}-x^{\ast}\Vert_{2}\geq|x_{i_{n_j}}^{n_j}|\geq\min_{i\in I^{n_j}%
}|x_{i}^{n_j}|\geq \eta_{\mu,q},
\]
which contradicts to $\Vert x^{n_j}-x^{\ast}\Vert_{2}<\eta_{\mu,q}.$

\textit{Case 2:} $I^{n_j}\neq I$ and $I^{n_j}\cap I=I^{n_j}.$
In this case, it is
obvious that $I^{n_j}\subset I$.
Thus, there exists an $i^{\ast}$ such that
$i^{\ast}\in I\setminus I^{n_j}$.
By Lemma {\ref{Lemma_OptimalCondition}}(a), we still have
\[
\Vert x^{n_j}-x^{\ast}\Vert_{2}\geq|x_{i^{\ast}}^{\ast}|\geq\min_{i\in
I}|x_{i}^{\ast}|\geq \eta_{\mu,q},
\]
and it contradicts to $\Vert x^{n_j}-x^{\ast}\Vert_{2}<\eta_{\mu,q}$.

Thus we have justified that $I^{n_j}=I$ when $j\geq j_0$.
Similarly, it can be also claimed that $I^{n+1} = I^n$ whenever $n>n^*$.
Therefore, whenever $n>n^*$, it holds $I^n = I$.

As $I^{n}=I$
when $n>n^*$, it suffices to test that $sgn(x_{i}^{(n)})=sgn(x_{i}^{\ast})$ for any $i\in I$.
Similar to the first part of proof, we will first check that $sgn(x_i^{n_j}) = sgn(x_i^*)$,
and then $sgn(x_i^{n+1}) = sgn(x_i^n)$ for any $i\in I$ by contradiction.
We now prove $sgn(x_i^{n_j}) = sgn(x_i^*)$ for any $i\in I$.
Assume this is not the case. Then there exists an
$i^{\ast}\in I$ such that $sgn(x_{i^*}^{n_j})\neq sgn(x_{i^*}^{\ast})$, and
hence,
\[
sgn(x_{i^{\ast}}^{n_j})sgn(x_{i^{\ast}}^{\ast})=-1.
\]
From Lemma {\ref{Lemma_OptimalCondition}}(a) and Theorem {\ref{Thm_SupportConv}}(a), it then implies
\begin{align*}
\Vert x^{n_j}-x^{\ast}\Vert_{2}
& \geq|x_{i^{\ast}}^{n_j}-x_{i^{\ast}}^{\ast}|=|x_{i^{\ast}}^{n_j}|+|x_{i^{\ast}}^{\ast}|\nonumber\\
& \geq\min_{i\in I}\{|x_{i}^{n_j}|+|x_{i}^{\ast}|\}\geq 2\eta_{\mu,q},
\end{align*}
contradicting again to $\Vert x^{n_j}-x^{\ast}\Vert_{2}<\eta_{\mu,q}$. This contradiction shows $sgn(x^{n_j})=sgn(x^{\ast})$.
Similarly, we can also show that $sgn(x^{n+1}) = sgn(x^n)$ whenever $n>n^*$.
Therefore, $sgn(x^n) = sgn(x^*)$ when $n>n^*$.

With this, the proof of Theorem {\ref{Thm_SupportConv}} is completed.
\end{proof}

In order to prove the convergence of the whole sequence, we do some modifications of the original sequence $\{x^n\}$,
and then yield a new sequence $\{u^n\}$ such that both sequences have the same convergence behaviours.
We describe these modifications as follows:
\begin{enumerate}
\item[(a)]
Let $n_0 = j_0N>n^*$ for some positive integer $j_0$.
Then we can define a new sequence $\{\hat{x}^n\}$ with $\hat{x}^n = x^{n_0+n}$ for ${n\in \mathbf{N}}$.
It is obvious that $\{\hat{x}^n\}$ has the same convergence behaviour with $\{x^n\}$.
Moreover, it can be noted from Theorem {\ref{Thm_SupportConv}} that all the support sets and signs of $\{\hat{x}^n\}$ are the same.

\item[(b)] Denote $I$ as the convergent support set of the sequence $\{x^n\}$. Let $K$ be the number of elements of $I$.
Without loss of generality, we assume
\[
1\leq I(1)<I(2)<\cdots < I(K) \leq N.
\]
According to the updating rule (\ref{updatingindex})-(\ref{updatingrul2}) of the CCD algorithm,
we can observe that many successive iterations of $\{\hat{x}^n\}$ are the same.
Thus, we can merge these successive iterations into a single iteration.
Moreover, the updating rule of the index is cyclic and thus periodic.
As a consequence, the merging procedure can be repeated periodically.
Formally, we consider such a periodic subsequence with $N$-length of $\{\hat{x}^n\}$, i.e.,
$$\{\hat{x}^{jN+I(1)},\hat{x}^{jN+I(1)+1}, \cdots, \hat{x}^{jN+I(1)+N-1}\}$$
for $j\in \mathbf{N}$.
Then for any $j\in \mathbf{N}$, we emerge the $N$-length sequence $\{\hat{x}^{jN+I(1)}, \cdots, \hat{x}^{jN+I(1)+N-1}\}$
into a new $K$-length sequence $\{\bar{x}^{jK+1},\bar{x}^{jK+2}, \cdots, \bar{x}^{jK+K}\}$ with the rule
$$
\{\hat{x}^{jN+I(i)},\cdots, \hat{x}^{jN+I(i+1)-1}\} \mapsto \bar{x}^{jK+i},
$$
with $\bar{x}^{jK+i}=\hat{x}^{jN+I(i)}$ for $i=1,2,\cdots,K,$
since $\hat{x}^{jN+I(i)+k} = \hat{x}^{jN+I(i)}$ for $k=1,\cdots,I(i+1)-I(i)-1.$
Moreover, we emerge the first $I(1)$ iterations of $\{\hat{x}^n\}$ into $\bar{x}^0$, i.e.,
$$
\{\hat{x}^{0},\cdots, \hat{x}^{I(1)-1}\} \mapsto \bar{x}^{0},
$$
with $\bar{x}^{0} = \hat{x}^{0}$, since these iterations keep invariant and are equal to $\hat{x}^{0}$.
After this procedure, we obtain a new sequence $\{\bar{x}^{n}\}$ with $n=jK+i$, $i=0,\cdots,K-1$ and $j\in \mathbf{N}$.
It can be observed that such an emerging procedure keeps the convergence behaviour of $\{\bar{x}^n\}$ the same as that of $\{\hat{x}^n\}$
and $\{x^n\}$.

\item[(c)] Furthermore, for the index set $I$, we define a projection $P_I$ as
$$
P_I: \mathbf{R^N} \rightarrow \mathbf{R}^K, P_I x = x_I, \forall x\in \mathbf{R}^N,
$$
where $x_I$ represents the subvector of $x$ restricted to the index set $I$.
With this projection, a new sequence $\{u^n\}$  is constructed such that
$$
u^n = P_I \bar{x}^n,
$$
for $n\in \mathbf{N}$.
As we can observe that $u^n$ keeps all the non-zero elements of $\bar{x}^n$ while gets rid of its zero elements.
Moreover, this operation can not change the convergence behavior of $\{\bar{x}^n\}$ and $\{u^n\}$.
Therefore, the convergence behaviour of $\{u^n\}$ is the same as $\{x^n\}$.
\end{enumerate}

In the following, we will prove the convergence of $\{x^n\}$ via justifying the convergence of $\{u^n\}$.
Let
$$\mathcal{U} = \{u^*: u^* = P_I x^*, \forall x^* \in \mathcal{L}\}.$$
Then $\mathcal{U}$ is the corresponding limit point set of $\{u^n\}$.
Furthermore, we define a new function $T$ as follows:
\begin{equation}
T: \mathbf{R}^K \rightarrow \mathbf{R}, T(u) = T_{\lambda}(P_I^T u), \forall u\in \mathbf{R}^K,
\label{def-T}
\end{equation}
where $P_I^T$ denotes the transpose of the projection $P_I$, and is defined as
\[
P_I^T: \mathbf{R}^K \rightarrow \mathbf{R}^N, (P_I^T u)_I = u,  (P_I^T u)_{I^c} = 0, \forall u \in \mathbf{R}^K,
\]
where $I^c$ represents the complementary set of $I$, i.e., $I^c = \{1,2,\cdots, N\} \setminus I$,
$(P_I^T u)_I$  and $(P_I^T u)_{I^c}$ represent the subvectors of $P_I^T u$ restricted to $I$ and $I^c$, respectively.
Let $B=A_I$, where $A_I$ denotes the submatrix of $A$ restricted to the index set $I$.
Thus,
\[
T(u) = \frac{1}{2} \|Bu-y\|_2^2 + \lambda \|u\|_q^q.
\]

After the modifications (a)-(c), we can observe that the following properties still hold for $\{u^n\}$.

\begin{lemma}\label{Lemma_propofnewsequence}
The sequence $\{u^n\}$ possesses the following properties:
\begin{enumerate}
\item[(a)]
$\{u^n\}$ is updated via the following cyclic coordinate descent rule.
Given the current iteration $u^n$, only the $i$-th coordinate will be updated while the other coordinate coefficients will be fixed at the next iteration, i.e.,
\begin{equation}
u^{n+1}_i= \mathcal{T} (v_i^n, u_i^n),
\label{newupdatingrul1}
\end{equation}
and
\begin{equation}
u_j^{n+1} = u_j^{n}, \ {\rm for}\ j\neq i,
\label{newupdatingrul2}
\end{equation}
where $i$ is determined by
\begin{equation}
  i= \left\{
  \begin{array}{cc}
  K   & {\rm if}\  0\equiv {(n+1)}\  {\rm mod} \ K \\
  {(n+1)}\  {\rm mod} \ K, & {\rm otherwise} %
  \end{array}%
  \right.,
  \label{newupdatingindex}
\end{equation}%
and
\begin{equation}
v_i^n = u_i^n - \mu B_i^T(Bu^n-y),
\label{newupdateFowStep}
\end{equation}

\item[(b)]
According to the updating rule (\ref{newupdatingrul1})-(\ref{newupdateFowStep}), for $n\geq K$,
then there exit two positive integers $1\leq i_0 \leq K$ and $j_0 \geq 1$ such that $n=j_0K+i_0$ and
\begin{equation}
  u_j^n= \left\{
  \begin{array}{cc}
  u_j^{n-(i_0-j)},   & {\rm if}\  1\leq j \leq i_0\\
  u_j^{n-K-(i_0-j)}, & {\rm if}\ i_0+1 \leq j \leq K %
  \end{array}%
  \right..
  \label{cyclic-u}
\end{equation}%

\item[(c)]
For any $n \in \mathbf{N}$,
\[
u^n \in \mathbf{R}_{{\eta_{\mu,q}}^c}^K,
\]
where $\mathbf{R}_{{\eta_{\mu,q}}^c}$ represents a one-dimensional real subspace, which is defined as
\[
\mathbf{R}_{{\eta_{\mu,q}}^c} = \mathbf{R} \setminus (-\eta_{\mu,q}, \eta_{\mu,q}).
\]

\item[(d)] Given $u^n$, and $i$ is determined by (\ref{newupdatingindex}), then $u_i^{n+1}$
satisfies the following equation
\begin{align}
&B_i^T(Bu^{n+1}-y) + \lambda q sgn(u_i^{n+1}) |u_i^{n+1}|^{q-1} \nonumber\\
&= (\frac{1}{\mu}-B_i^TB_i)(u_i^n - u_i^{n+1}).
\label{newOptCon_i}
\end{align}
That is,
$$
\nabla_i T(u^{n+1}) = (\frac{1}{\mu}-B_i^TB_i)(u_i^n - u_i^{n+1}),
$$
where $\nabla_i T(u^{n+1})$ represents the gradient of $T(\cdot)$ with respect to the $i$-th coordinate at the point $u^{n+1}$.

\item[(e)]
$\{u^n\}$ satisfies the following sufficient decrease condition:
%Assume that $0<\mu<L_{\max}^{-1}$, then
\[
T(u^n) - T(u^{n+1}) \geq a \|u^n - u^{n+1}\|_2^2,
\]
for $n\in \mathbf{N}$, where $a=\frac{1}{2}(\frac{1}{\mu} - L_{\max}).$

\item[(f)]
\[
\|u^{n+1} -u^n\|_2 \rightarrow 0, \ {\rm as} \ n\rightarrow \infty.
\]

\end{enumerate}
\end{lemma}

\begin{proof}
The properties of $\{u^n\}$ listed in Lemma {\ref{Lemma_propofnewsequence}} are some direct extensions of those of $\{x^n\}$.
More specifically, Lemma {\ref{Lemma_propofnewsequence}}(a) can be derived by the CCD algorithm updating rule (\ref{updatingindex})-(\ref{updatingrul2}) and the modification procedure.
Lemma {\ref{Lemma_propofnewsequence}}(b) is obtained directly by the cyclic updating rule.
Lemma {\ref{Lemma_propofnewsequence}}(c) and (d) can be derived by Property {\ref{Prop_OptimalCondition_n+1}}(b) and the updating rule (\ref{newupdatingrul1})-(\ref{newupdateFowStep}).
Lemma {\ref{Lemma_propofnewsequence}}(e) can be obtained by Property {\ref{Prop_SufficientDecrease}} and the definition of $T$ (\ref{def-T}).
Lemma {\ref{Lemma_propofnewsequence}}(f) can be directly derived by Property {\ref{Prop_AsymptoticRegular}}.
\end{proof}

Besides Lemma {\ref{Lemma_propofnewsequence}}, the following lemma shows that the gradient sequence $\{\nabla T(u^n)\}$ satisfies the so-called relative error condition {\cite{Attouch2013}}, which is useful for proving the convergence of $\{u^k\}$.
%and the second one is to demonstrate that the function $T$ satisfies the Kurdyka-{\L}ojasiewicz property {\cite{Attouch2013}}.

\begin{lemma}\label{Lemma_RelativeErrCond}
When $n\geq K-1$, $\nabla T(u^{n+1})$ satisfies
\begin{equation*}
\|\nabla T(u^{n+1})\|_2 \leq b \|u^{n+1} - u^n\|_2,
\end{equation*}
where $b = (\frac{1}{\mu} + K\delta)\sqrt{K},$ with
$$\delta = \max_{i,j= 1,2,\cdots,K} |B_i^TB_{j}|.$$
\end{lemma}

\begin{proof}
We assume that $n+1 = j^*K+i^*$ for some positive integers $j^* \geq 1$ and $1\leq i^*\leq K$.
For simplicity, let
\begin{equation}
i^*=K.
\label{i-equalto-K}
\end{equation}
If not, we can renumber the indices of the coordinates such that (\ref{i-equalto-K}) holds while the iterative sequence $\{u^n\}$ keeps invariant,
since the updating rule (\ref{newupdatingindex}) is cyclic and thus periodic.
Such an operation can be described as follows:
for each $n\geq K$, by Lemma {\ref{Lemma_propofnewsequence}}(b), we know that the coefficients of $u^n$ are only related to the previous $K-1$ iterates.
Thus, we consider the following a period of the original updating order, i.e.,
\[
\{i^*+1,\cdots,K,1,\cdots,i^*\},
\]
then we can renumber the above coordinate updating order as
\[
\{1',\cdots,(K-i^*)',(K-i^*+1)',\cdots,K'\},
\]
with
\begin{equation*}
  j'= \left\{
  \begin{array}{cc}
  i^*+j,   & {\rm if}\  1\leq j \leq K-i^*\\
  j-(K-i^*), & {\rm if}\ K-i^*< j \leq K %
  \end{array}%
  \right..
  \label{renum-j}
\end{equation*}%

In the following, we will calculate $\nabla_i T(u^{n+1})$ by a recursive way for $i=K,K-1,\cdots,1$. Specifically,
\begin{enumerate}
\item[(a)]
For $i=K$, by Lemma {\ref{Lemma_propofnewsequence}}(d), it holds
\begin{equation}
\nabla_K T(u^{n+1}) = (\frac{1}{\mu} - B_K^TB_K) (u_K^n - u_K^{n+1}).
\label{nablaT-n+1-K}
\end{equation}
For any $i=K-1,K-2,\cdots,1,$
\begin{equation*}
\nabla_i T(u^{n+1}) = B_i^T(Bu^{n+1}-y) + \lambda q sgn(u_i^{n+1})|u_i^{n+1}|^{q-1},
\end{equation*}
and $u_i^{n+1} = u_i^n$.
Therefore, for $i=K-1,K-2,\cdots,1,$
\begin{equation}
\nabla_i T(u^{n+1}) = \nabla_i T(u^n) + B_i^T B_K (u_K^{n+1} - u_K^n).
\label{nablaT-red1}
\end{equation}

\item[(b)]
For $i=K-1$, since $n=j^*K+(K-1)$, then by Lemma {\ref{Lemma_propofnewsequence}}(d) again, it holds
\begin{equation}
\nabla_{K-1} T(u^{n}) = (\frac{1}{\mu} - B_{K-1}^TB_{K-1}) (u_{K-1}^{n-1} - u_{K-1}^{n}).
%\label{nablaT-n-K-1}
\end{equation}
By Lemma {\ref{Lemma_propofnewsequence}}(b), it implies
$$
u_{K-1}^{n-1} = u_{K-1}^{n+1}.
$$
Thus,
\begin{equation}
\nabla_{K-1} T(u^{n}) = (\frac{1}{\mu} - B_{K-1}^TB_{K-1}) (u_{K-1}^{n+1} - u_{K-1}^{n}).
\label{nablaT-n-K-1}
\end{equation}
Combing (\ref{nablaT-red1}) with (\ref{nablaT-n-K-1}),
\begin{align}
&\nabla_{K-1} T(u^{n+1})= \nonumber\\
& (\frac{1}{\mu} - B_{K-1}^TB_{K-1}) (u_{K-1}^{n+1} - u_{K-1}^{n}) \nonumber\\
&+ B_{K-1}^T B_K (u_K^{n+1} - u_K^n)
\label{nablaT-n+1-K-1}
\end{align}
Similarly to (\ref{nablaT-red1}), for $i=K-2,K-3,\cdots,1$, we have
\begin{equation}
\nabla_i T(u^{n}) = \nabla_i T(u^{n-1}) + B_i^T B_{K-2} (u_{K-2}^{n} - u_{K-2}^{n-1}).
\label{nablaT-red2}
\end{equation}

\item[(c)]
For any $i=K-j$ with $0\leq j \leq K-1$, by a recursive way, we have
\begin{align}
&\nabla_{K-j} T(u^{n+1}) \nonumber\\
&= \nabla_{K-j} T(u^n) + B_{K-j}^TB_K(u_K^{n+1} -u_K^n)\nonumber\\
&= \nabla_{K-j} T(u^{n-1}) + B_{K-j}^T \sum_{k=0}^{1} B_{K-k}(u_{K-k}^{n+1-k} -u_{K-k}^{n-k})\nonumber\\
&= \cdots \nonumber\\
&= \nabla_{K-j} T(u^{n-j+1}) \nonumber\\
&+ B_{K-j}^T \sum_{k=0}^{j-1} B_{K-k}(u_{K-k}^{n+1-k} -u_{K-k}^{n-k}).
\label{nablaT-n+1-K-j}
\end{align}
Moreover, Lemma {\ref{Lemma_propofnewsequence}}(d) gives
\begin{align}
&\nabla_{K-j} T(u^{n-j+1}) \nonumber\\
&= (\frac{1}{\mu} - B_{K-j}^T B_{K-j}) (u_{K-j}^{n-j} - u_{K-j}^{n-j+1}).
\label{nablaT-n-j+1-K-j}
\end{align}
Plugging (\ref{nablaT-n-j+1-K-j}) into (\ref{nablaT-n+1-K-j}), it holds
\begin{align}
&\nabla_{K-j} T(u^{n+1}) = \frac{1}{\mu} (u_{K-j}^{n-j} - u_{K-j}^{n-j+1})\nonumber\\
&  + \sum_{k=0}^{j} B_{K-j}^T B_{K-k}(u_{K-k}^{n+1-k} -u_{K-k}^{n-k}),
\label{nablaT-K-j}
\end{align}
for $j=0,1,\cdots,K-1.$
Furthermore, by Lemma {\ref{Lemma_propofnewsequence}}(b), it implies
$$
u_{K-k}^{n+1-k} = u_{K-k}^{n+1}
$$
and
$$
u_{K-k}^{n-k} = u_{K-k}^{n}
$$
for $0\leq k \leq K-1$.
Therefore, (\ref{nablaT-K-j}) becomes
\begin{align}
&\nabla_{K-j} T(u^{n+1}) = \frac{1}{\mu} (u_{K-j}^{n} - u_{K-j}^{n+1}) \nonumber\\
&+ \sum_{k=0}^{j} B_{K-j}^T B_{K-k}(u_{K-k}^{n+1} -u_{K-k}^{n}),
\label{nablaT-K-j*}
\end{align}
for $j=0,1,\cdots,K-1.$
\end{enumerate}

Furthermore, by (\ref{nablaT-K-j*}), it implies
\begin{align}
&|\nabla_{K-j} T(u^{n+1})| \leq \frac{1}{\mu} |u_{K-j}^{n} - u_{K-j}^{n+1}| \nonumber\\
&+ \sum_{k=0}^{j} (|B_{K-j}^T B_{K-k}| \cdot |u_{K-k}^{n+1} -u_{K-k}^{n}|)\nonumber\\
& \leq \frac{1}{\mu} |u_{K-j}^{n} - u_{K-j}^{n+1}| + \delta \|u^{n+1}-u^n\|_1,
\label{nablaT-K-j-abs}
\end{align}
for $j=0,1,\cdots,K-1,$
where the second inequality holds for
$$\delta = \max_{i,j=1,\cdots,K} |B_i^T B_j|$$
and
$$\sum_{k=0}^{j} |u_{K-k}^{n+1} -u_{K-k}^{n}| \leq \|u^{n+1}-u^n\|_1.$$
Summing $|\nabla_{K-j} T(u^{n+1})|$ with respect to $j$ gives
\begin{align}
&\|\nabla T(u^{n+1})\|_1 \nonumber\\
&\leq \frac{1}{\mu} \|u^{n+1}-u^n\|_1 + K\delta \|u^{n+1}-u^n\|_1 \nonumber\\
&\leq (\frac{1}{\mu} + K\delta)\sqrt{K}\|u^{n+1}-u^n\|_2,
\label{nablaT-1-norm}
\end{align}
where the second inequality holds for the norm inequality between 1-norm and 2-norm,
that is,
\begin{align}
\|u\|_2 \leq \|u\|_1 \leq \sqrt{K} \|u\|_2,
\label{NormIneqn}
\end{align}
for any $u\in \mathbf{R}^K$.
Also, combing (\ref{NormIneqn}) and (\ref{nablaT-1-norm}) implies
\begin{equation*}
\|\nabla T(u^{n+1})\|_2 \leq (\frac{1}{\mu} + K\delta)\sqrt{K}\|u^{n+1}-u^n\|_2.
\end{equation*}
\end{proof}

Furthermore, according to {\cite{Attouch2013}} (p. 122), we know that the function
$$T(u) = \frac{1}{2}\|Bu-y\|_2^2 + \lambda \|u\|_q^q$$
is a {\bf Kurdyka-{\L}ojasiewicz (KL)} function with a desingularizing function of the form $\varphi(s) = cs^{\theta},$
where $c>0$, $\theta \in [0,1).$
As a consequence, we can claim the following lemma.
\begin{lemma}\label{Lemma_KL_T}
For any $u^* \in \mathbf{R}_{{\eta_{\mu,q}}^c}^K$ (where $\mathbf{R}_{{\eta_{\mu,q}}^c}^K$ is defined as in Lemma {\ref{Lemma_propofnewsequence}}(c)),
there exist a neighborhood $U$ of $u^*$ and a constant $\xi>0$ such that for all $u\in U\cap \{u: T(u^*)<T(u)<T(u^*) + \xi\}$, it holds
\begin{equation}
\varphi'(T(u)-T(u^*)) \|\nabla T(u)\|_2\geq 1.
\label{KLIneq}
\end{equation}
\end{lemma}

With Lemmas {\ref{Lemma_propofnewsequence}}-{\ref{Lemma_KL_T}}, we can prove the convergence of $\{u^n\}$ as the following theorem.

\begin{theorem}\label{Thm_Conv_un}
The sequence $\{u^n\}$ is convergent.
\end{theorem}

\begin{proof}
Assume that $u^* \in \mathcal{U}$ is a limit point of $\{u^n\}$.
By Lemma {\ref{Lemma_propofnewsequence}}, we have known the following facts:
\begin{enumerate}
\item[(i)]
$\|u^{n+1} - u^n\|_2 \rightarrow 0$ as $n\rightarrow \infty$;

\item[(ii)]
$\{T(u^n)\}$ is monotonically decreasing and converges to $T(u^*)$;

\item[(iii)]
there exists a subsequence $\{u^{n_j}\}$ converges to $u^*$, that is,
$$
u^{n_j} \rightarrow u^* \ \text{as} \ j\rightarrow \infty.
$$
\end{enumerate}
Therefore, for any positive constant $\varepsilon<\xi$,
there exists a sufficiently large integer $j^*>0$ such that
when $n\geq n_{j^*}$,
\begin{equation}
\|u^{n+1}-u^n\|_2 <\varepsilon \ \text{and}\ 0< T(u^n) - T(u^*) < \varepsilon,
\label{Con1}
\end{equation}
and when $j\geq j^*$,
\begin{equation}
\|u^{n_{j}} - u^*\|_2 < \varepsilon,
\label{Con2}
\end{equation}
and further by the fact that $T$ is a KL function with $\varphi$ as the desingularizing function,
\begin{equation}
\|u^*-u^{n_{j^*}}\|_2 + 2\sqrt{\frac{T(u^{n_{j^*}})-T(u^*)}{a}}+\frac{b}{a} \varphi(T(u^{n_{j^*}})-T(u^*)) < \varepsilon.
\label{InitialAss1}
\end{equation}

We redefine a new sequence $\{\hat{u}^n\}$ for $n\in \mathbf{N}$ with
\[
\hat{u}^n = u^{n+n_{j^*}}.
\]
Then the following inequalities hold naturally for each $n\in \mathbf{N}$,
\[
\|\hat{u}^{n+1} - \hat{u}^n\|_2 < \varepsilon \ \text{and}
\ 0<T(\hat{u}^n) - T(u^*) < \varepsilon,
\]
and
\begin{equation}
\|u^*-\hat{u}^0\|_2 + 2\sqrt{\frac{T(\hat{u}^0)-T(u^*)}{a}}+\frac{b}{a} \varphi(T(\hat{u}^0)-T(u^*)) < \varepsilon.
\label{InitialAss}
\end{equation}
Therefore, the convergence of $\{u^n\}$ is equivalent to the convergence of $\{\hat{u}^n\}$.

The key point to prove the convergence of $\{\hat{u}^n\}$ is to justify the following claim: for $n=1,2,\cdots$
\begin{equation}
\hat{u}^n \in \mathbf{B}(u^*,\varepsilon),
\label{NeiBor_z*}
\end{equation}
that is, $\hat{u}^n$ lies in an $\varepsilon$-neighborhood of $u^*$,
and
\begin{align}
&\sum_{i=1}^n \|\hat{u}^{i+1} - \hat{u}^i\|_2 + \|\hat{u}^{n+1} - \hat{u}^n\|_2 \leq
\|\hat{u}^0 - \hat{u}^1\|+\nonumber\\
& \frac{b}{a} \left(\varphi({T(\hat{z}^1)-T({z}^*)})-\varphi({T(\hat{z}^{n+1})-T({z}^*)})\right).
\label{CauchySeq1}
\end{align}

By Lemma {\ref{Lemma_propofnewsequence}}(e), it can be observed that
\begin{equation}
a\|\hat{u}^{n+1}-\hat{u}^n\|_2^2 \leq T(\hat{u}^n) - T(\hat{u}^{n+1})
\label{SuffDecr}
\end{equation}
for any $n\in \mathbf{N}$.
Fix $n\geq 1$, we claim that if $\hat{u}^n \in B(u^*,\varepsilon)$,
then
\begin{align}
&2\|\hat{u}^{n+1} - \hat{u}^n\|_2 \leq \|\hat{u}^n-\hat{u}^{n-1}\|_2 +\nonumber\\
& \frac{b}{a} \left(\varphi({T(\hat{u}^n) - T(u^*)}) - \varphi({T(\hat{u}^{n+1}) - T(u^*)})\right).
\label{MajorizingSeq}
\end{align}
Since $T$ is a KL function, by Lemma {\ref{Lemma_KL_T}}, it holds
\begin{equation}
\varphi'(T(\hat{u}^n) - T(u^*)) \geq \frac{1}{\|\nabla T(\hat{u}^{n})\|_2}.
\label{KLIneq1}
\end{equation}
Moreover, by Lemma {\ref{Lemma_RelativeErrCond}},
\begin{equation}
\varphi'(T(\hat{z}^n) - T(z^*)) \geq \frac{1}{\|\nabla T(\hat{u}^{n})\|_2} \geq \frac{1}{b\|\hat{u}^n - \hat{u}^{n-1}\|_2}.
\label{KLIneq2}
\end{equation}
Furthermore, by the concavity of the function $\varphi{(s)}$ for $s>0$,
\begin{align}
&\varphi({T(\hat{u}^n) - T(u^*)}) - \varphi({T(\hat{u}^{n+1}) - T(u^*)}) \nonumber\\
&\geq \varphi'(T(\hat{u}^n) - T(u^*))(T(\hat{u}^n)-T(\hat{u}^{n+1})).
\label{ConcavitySqrt}
\end{align}
Plugging the inequalities (\ref{SuffDecr}) and (\ref{KLIneq2}) into (\ref{ConcavitySqrt}) and after some simplifications,
\begin{align*}
&\|\hat{u}^n - \hat{u}^{n+1}\|_2^2 \leq \frac{b}{a} \|\hat{u}^n -\hat{u}^{n-1}\|_2 \nonumber\\
&\times \left( \varphi({T(\hat{u}^n) - T(u^*)}) - \varphi({T(\hat{u}^{n+1}) - T(u^*)}) \right).
\end{align*}
Using the inequality $\sqrt{\alpha \beta} \leq \frac{\alpha+\beta}{2}$ for any $\alpha, \beta \geq 0$,
we conclude that inequality (\ref{MajorizingSeq}) is satisfied.
Thus, the claim (\ref{CauchySeq1}) can be easily derived from (\ref{MajorizingSeq}).

In the following, we will prove the claim (\ref{NeiBor_z*}) by induction.

First, by (\ref{InitialAss}), it implies
$$
\hat{u}^0 \in B(u^*,\varepsilon).
$$

Second, it can be observed that
\begin{align*}
\|\hat{u}^1 - u^*\|_2
&\leq \|\hat{u}^1 - \hat{u}^0\|_2 + \|\hat{u}^0 - u^*\|_2 \nonumber\\
&\leq \sqrt{\frac{T(\hat{u}^0) - T(\hat{u}^1)}{a}} + \|\hat{u}^0 - u^*\|_2 \nonumber\\
&\leq \sqrt{\frac{T(\hat{u}^0) - T({u}^*)}{a}} + \|\hat{u}^0 - u^*\|_2 \nonumber\\
&<\varepsilon,
\end{align*}
where the second inequality holds for (\ref{SuffDecr}), the third inequality holds due to $T(u^*) \leq T(\hat{u}^1) \leq T(\hat{u}^0)$
and the last inequality holds for (\ref{InitialAss}).
Therefore, $\hat{u}^1 \in B(u^*,\varepsilon)$.

Third, suppose that $\hat{u}^n \in B(u^*,\varepsilon)$ for $n\geq 1$, then
%Suppose that (\ref{NeiBor_z*}) holds for $n\geq 1$, then
\begin{align}
&\|\hat{u}^{n+1} - u^*\|_2 \leq \|u^* - \hat{u}^0\|_2 + \|\hat{u}^0 - \hat{u}^1\|_2 + \sum_{i=1}^n \|\hat{u}^{i+1} - \hat{u}^i\|_2 \nonumber\\
&\leq \|u^* - \hat{u}^0\|_2  + 2 \|\hat{u}^0 - \hat{u}^1\|_2 \nonumber\\
&+ \frac{b}{a} \left( \varphi({T(\hat{u}^1)-T(u^*)}) - \varphi({T(\hat{u}^{n+1}) - T(u^*)}) \right),
\label{ThdStep}
\end{align}
where the second inequality holds for (\ref{CauchySeq1}).
Moreover,
\begin{equation}
\|\hat{u}^0 - \hat{u}^1\|_2 \leq \sqrt{\frac{T(\hat{u}^0) - T(\hat{u}^1)}{a}} \leq \sqrt{\frac{T(\hat{u}^0) - T({u}^*)}{a}},
\label{diffz0z1}
\end{equation}
where the first inequality holds for (\ref{SuffDecr}) and the second inequality holds for $T(u^*) \leq T(\hat{u}^1) \leq T(\hat{u}^0)$.
Also, it is obvious that
\begin{equation}
\varphi({T(\hat{z}^1)-T(z^*)}) - \varphi({T(\hat{z}^{n+1}) - T(z^*)}) \leq \varphi({T(\hat{z}^0)-T(z^*)}).
\label{diffz1z*}
\end{equation}
Plugging (\ref{diffz0z1}) and (\ref{diffz1z*}) into (\ref{ThdStep}) and using (\ref{InitialAss}),
we can claim that
$$
\hat{u}^{n+1} \in B(u^*,\varepsilon).
$$

By (\ref{CauchySeq1}), it shows that
\[
\sum_{i=1}^n \|\hat{u}^{i+1} - \hat{u}^i\|_2 \leq \|\hat{u}^0 - \hat{u}^1\|_2
+ \frac{b}{a}\varphi({T(\hat{u}^1)-T({u}^*)}).
\]
Therefore,
\[
\sum_{i=1}^{\infty} \|\hat{u}^{n+1} - \hat{u}^n\|_2 < +\infty,
\]
which implies that the sequence $\{\hat{u}^n\}$ converges to some $u^{**}$.
While we have assumed that there exists a subsequence $\{\hat{u}^{n_j}\}$ converges to $u^*$,
then it must hold
\[
u^{**} = u^*.
\]
Consequently, we can claim that $\{u^n\}$ converges to $u^*$.

%With these, we end the proof of Theorem 4.
\end{proof}

From Theorem {\ref{Thm_Conv_un}}, the sequence $\{u^n\}$ is convergent.
As a consequence, we can also claim the convergence of $\{x^n\}$ as shown in the following theorem.

\begin{theorem}\label{Thm_MainConv}
Let $\{x^n\}$ be a sequence generated by the CCD algorithm.
Assume that $0<\mu<L_{\max}^{-1}$ and $T_{\lambda}(x^0) < +\infty$, then $\{x^n\}$ converges to a stationary point.
\end{theorem}

\begin{proof}
The proof of Theorem {\ref{Thm_MainConv}} can be directly derived by Theorems {\ref{Thm_LimitPoint_StatPoint}}, {\ref{Thm_Conv_un}}
and the fact that the convergence behaviours $\{x^n\}$ and $\{u^n\}$ are the same.
\end{proof}

%By Lemma 2, it can be noted that under a more common and slight stricter condition, i.e., $0<\mu<\|A\|_2^{-2},$
%any global minimizer of the $l_q$ regularization problem is a stationary point.
%
%
%From Theorem 2, we note that the condition on $\mu$ becomes $(0,\|A\|_2^{-2})$,
%which is slightly stricter than $(0,L_{\max}^{-1})$.
%This stricter condition is required to guarantee that any global minimizer of the $l_q$ regularization problem belongs to the fixed point set ${\mathcal{F}}_q$ by Lemma 2.

\subsection{Convergence to A Local Minimizer}

In this subsection, we mainly answer the second open question proposed in the end of the subsection III.A.
More specifically, we will demonstrate that the CCD algorithm converges to a local minimizer of the $l_q$ regularization problem under certain conditions.

\begin{theorem}\label{Thm_LocalMin}
Let $\{x^n\}$ be a sequence generated by the CCD algorithm.
Assume that $0<\mu<L_{\max}^{-1}$, $T_{\lambda}(x^0) < +\infty$, and $x^*$ is a convergent point of $\{x^n\}$.
Let $I = Supp(x^*)$, and $K=\|x^*\|_0$. Then $x^*$ is a local minimizer of $T_{\lambda}$ if the following conditions hold:
\begin{enumerate}
\item[(a)]
$\sigma_{\min}(A_I^TA_I)>0$;
\item[(b)]
$0<\lambda < \frac{\sigma_{\min}(A_I^TA_I)|e|^{2-q}}{q(1-q)}$,
where $e = \min_{i\in I} |x_i^*|$.
\end{enumerate}
\end{theorem}

Intuitively, under the condition (b) in Theorem {\ref{Thm_LocalMin}},
it follows that the principle submatrix of the Henssian matrix of $T_{\lambda}$ at $x^*$ restricted to the index set $I$ is positively definite.
Thus, the convexity of the objective function can be guaranteed in a neighborhood of $x^*$.
As a consequence, $x^*$ should be a local minimizer.

\begin{proof}
For simplicity, let
$$
F(x) = \frac{1}{2} \|Ax-y\|_2^2,
\Phi(x) = \sum_{i=1}^N \phi(x_i)
$$
with
$$\phi(x_i) = |x_i|^q.$$
Thus, $T_{\lambda}(x) = F(x) + \lambda \Phi(x)$, and for $v\neq 0,$
\begin{equation}
\phi'(v) = q sgn(v)|v|^{q-1},  \phi''(v) = q(q-1)|v|^{q-2}.
\label{Derivative-phi}
\end{equation}

Let
\begin{equation}
\epsilon = \frac{1}{2} \left(\sigma_{\min}(A_I^TA_I)+\lambda \phi''(e)\right).
\label{epsilon}
\end{equation}
By the assumption of Theorem {\ref{Thm_LocalMin}}, it holds $\epsilon>0$.
Furthermore, we define two constants
\begin{align}
C_0= \max{\{\|A_I^TA_I\|_1, 2\|A_I^TA_{I^c}\|_1\}},
\label{C_epsilon}
\end{align}
and
\begin{align*}
C = \frac{\tau_{\mu}}{\lambda\mu} + \frac{\sqrt{N}e}{\lambda}C_0,
\end{align*}
where $e=\min_{i\in I} |x_i^*|$.
In the following, we will show that there exists a constant $0<c<1$, it holds
$$
T_{\lambda}(x^*+h) - T_{\lambda}(x^*) \geq 0,
$$
whenever $\|h\|_2<ce$.

Actually, we have
\begin{align}
&T_{\lambda}(x^*+h) - T_{\lambda}(x^*) = F(x^*+h) - F(x^*) \nonumber\\
& + \lambda \left(\sum_{i\in I} (\phi(x_i^*+h_i) - \phi(x_i^*))+ \sum_{i\in I^c}\phi(h_i)\right).
\label{ErrEqn}
\end{align}
By Taylor expansion, it holds
\begin{align}
&F(x^*+h) - F(x^*)\nonumber\\
&= \langle h_I, A_I^T(Ax^* -y) \rangle + \langle h_{I^c}, A_{I^c}^T(Ax^* -y) \rangle \nonumber\\
&+ \frac{1}{2} \left( h_I^T A_I^TA_I h_I +  h_{I^c}^T A_{I^c}^TA_{I^c} h_{I^c} \right)
+  h_{I^c}^T A_{I^c}^TA_{I}h_I,
\label{TaylorExpF}
\end{align}
and
\begin{align}
&\sum_{i\in I} (\phi(x_i^*+h_i) - \phi(x_i^*))= \nonumber\\
& \sum_{i\in I} (\phi'(x_i^*)h_i  + \frac{1}{2} \phi''(x_i^*)|h_i|^2 + o(|h_i|^2)).
\label{TaylorExpPhi}
\end{align}
%for some $\xi_i \in (0,1)$, $i\in I.$
Denote $\Lambda_1\in \mathbf{R}^{K\times K}$ as a diagonal matrix generated by $\{\phi''(x_i^*)\}_{i\in I}$,
that is, for $i=1,2,\cdots,K$,
\begin{equation}
\Lambda_1(i,i) = \phi''(x_{I(i)}^*) = q(q-1) |x_{I(i)}^*|^{q-2},
\end{equation}
where $I(i)$ represents the $i$-th element of $I.$
By Lemma {\ref{Lemma_OptimalCondition}}(b), it implies
\begin{equation}
A_i^T(Ax^*-y) + \lambda \phi'(x_i^*) = 0.
\label{OptCond}
\end{equation}
for $i\in I$.
Plugging (\ref{TaylorExpF}), (\ref{TaylorExpPhi}) and (\ref{OptCond}) into (\ref{ErrEqn}), it becomes
\begin{align}
&T_{\lambda}(x^*+h) - T_{\lambda}(x^*) \nonumber\\
&= \frac{1}{2}h_I^T \left( A_I^TA_I+ \lambda \Lambda_1\right) h_I + o(\|h_I\|_2^2) \nonumber\\
& + \frac{1}{2} h_{I^c}^T \left(A_{I^c}^TA_{I^c} h_{I^c} + 2A_{I^c}^TA_I h_I\right) \nonumber\\
&+ h_{I^c}^T A_{I^c}^T(Ax^*-y)+ \lambda \sum_{i\in I^c} \phi(h_i).
\label{ErrEqn1}
\end{align}
Moreover, by the definition of $o(\|h_I\|_2^2)$,
there exists a constant $0<c_{\epsilon}<1$ (depending on $\epsilon$) such that
$|o(\|h_I\|_2^2)|\leq \epsilon \|h_I\|_2^2$ whenever $\|h_I\|_2 < c_{\epsilon} e$.
Therefore
\begin{align}
&T_{\lambda}(x^*+h) - T_{\lambda}(x^*) \nonumber\\
&= \frac{1}{2}h_I^T \left( A_I^TA_I+ \lambda \Lambda_1\right) h_I - \epsilon \|h_I\|_2^2 \nonumber\\
& + \frac{1}{2} h_{I^c}^T \left(A_{I^c}^TA_{I^c} h_{I^c} + 2A_{I^c}^TA_I h_I\right) \nonumber\\
&+ h_{I^c}^T A_{I^c}^T(Ax^*-y)+ \lambda \sum_{i\in I^c} \phi(h_i).
\label{ErrEqn2}
\end{align}
Furthermore, we divide the right side of the inequality (\ref{ErrEqn2}) into three parts, that is, $E_1$, $E_2$ and $E_3$ with
\begin{align}
E_1 = \frac{1}{2}h_I^T \left( A_I^TA_I+ \lambda \Lambda_1\right) h_I - \epsilon \|h_I\|_2^2,
\label{E1}
\end{align}
\begin{align}
E_2 = \frac{1}{2} h_{I^c}^T \left(A_{I^c}^TA_{I^c} h_{I^c} + 2A_{I^c}^TA_I h_I\right),
%\left(\nabla_{{I^c}{I^c}}^2 F(x^*) h_{I^c} + \left(\nabla_{{I^c}{I}}^2 F(x^*)+ \left(\nabla_{{I}{I^c}}^2 F(x^*)\right)^T\right)h_I\right) - \epsilon \|h_{I^c}\|_2^2,
\label{E2}
\end{align}
\begin{align}
E_3 = h_{I^c}^T A_{I^c}^T(Ax^*-y)+ \lambda \sum_{i\in I^c} \phi(h_i).
\label{E3}
\end{align}

By the definition of $\epsilon$ as in (\ref{epsilon}), it can be observed that
\begin{align*}
\sigma_{\min}(A_I^TA_I+\lambda \Lambda_1) \geq \sigma_{\min}(A_I^TA_I) + \lambda \phi''(e),
\end{align*}
and thus,
\begin{align}
E_1 &\geq \left(\frac{1}{2} \sigma_{\min}
(A_I^TA_I+\lambda \Lambda_1)-\epsilon\right)\|h_I\|_2^2 \nonumber\\
&\geq 0.
\label{E11}
\end{align}
Since $0<q<1$, then for any $v\in \mathbf{R}_{+}$,
it holds $\phi'(v)=qsgn(v)|v|^{q-1} \rightarrow \infty$ as $v\rightarrow 0^+$,
%which implies that for any sufficiently large number $C>0$,
which implies that there exists a sufficiently small constant $0<c<\frac{c_{\epsilon}}{\sqrt{N}}$
such that $\phi'(v) > C$ for any $0<v<ce$.
By Taylor expansion and $\phi(0) = 0$,
we have
\begin{equation*}
\phi(v) \geq C v>0,
\label{I12}
\end{equation*}
when $0<v<ce$.
Similarly, it holds
\begin{equation*}
\phi(v) \geq -C v>0,
\label{I12}
\end{equation*}
when $-ce<v<0$.
Thus, when $0<|v|<ce$,
\begin{equation}
\phi(v) \geq C|v|.
\label{phi_lowBound}
\end{equation}
Moreover,
by Lemma {\ref{Lemma_OptimalCondition}}(c), for any $i\in I^c$,
$$
\phi(h_i) + \frac{1}{\lambda}A_i^T(Ax^*-y)h_i
\geq \phi(h_i) - \frac{\tau_{\mu,q}}{\lambda\mu}|h_i|
\geq (C- \frac{\tau_{\mu,q}}{\lambda\mu})|h_i|.
$$
Thus,
\begin{align*}
E_3\geq \lambda (C - \frac{\tau_{\mu}}{\lambda\mu})\|h_{I^c}\|_1
= {\sqrt{N}e}C_0 \|h_{I^c}\|_1.
\end{align*}
Moreover, since $\|h\|_2<c_{\epsilon}e$ with $0<c_{\epsilon}<1$, then
$\|h\|_1< \sqrt{N}e$, $\|h_I\|_1< \sqrt{N}e$ and $\|h_{I^c}\|_1<\sqrt{N}e$.
It is easy to check that
\begin{align*}
|E_2|
&\leq \frac{1}{2} \|h_{I^c}\|_1 \|A_{I^c}^TA_{I^c}\|_1 \|h_{I^c}\|_1
+ \|h_{I^c}\|_1 \|A_{I^c}^TA_I\|_1 \|h_I\|_1 \nonumber\\
&\leq \|h_{I^c}\|_1 C_0 \|h\|_1
\leq \sqrt{N}e C_0 \|h_{I^c}\|_1 \leq E_3,
\end{align*}
where the second inequality holds for $\|h_{I^c}\|_1 \leq  \|h\|_1$,  $\|h_{I}\|_1 \leq  \|h\|_1$ and the definition of $C_0$ as specified in (\ref{C_epsilon}).
%\begin{align*}
%{\sqrt{N}e}C_{\epsilon} \|h_{I^c}\|_1 \geq \|h_{I^c}\|_1 C_{\epsilon}\|h\|_1  \geq |E_2|.
%\end{align*}
It implies that
\begin{equation}
E_2 + E_3 \geq 0.
\label{E2E3}
\end{equation}
By (\ref{E11}) and (\ref{E2E3}), it holds
\begin{equation*}
T_{\lambda}(x^*+h) - T_{\lambda}(x^*) \geq E_1 + E_2 + E_3 \geq 0,
\label{TotalDiff}
\end{equation*}
for any sufficiently small $h$.
Therefore, $x^*$ is a local minimizer of $T_{\lambda}$.

\end{proof}

\section{Numerical Experiments}

In this section, we demonstrate the effects of the algorithmic parameters on the performance of the proposed CCD algorithm.
Particularly, we will mainly focus on the effect of the stepsize parameter,
while the effects of the regularization parameter $\lambda$ and $q$ can be referred to {\cite{lqCD-Marjanovic-2014}}.

For this purpose, we consider the performance of the CCD algorithm for the sparse signal recovery problem (\ref{ObsMode2}).
In these experiments, we set $m=200,$ $N=400$ and $k^* =20,$
where $m$ is the number of measurements,
$N$ is the dimension of signal and $k^*$ is the sparsity level of the original sparse signal.
The original sparse signal $x^*$ is generated randomly according to the standard Gaussian distribution.
$A$ is of dimension $m\times N=200\times400$ with Gaussian $\mathcal{N}(0,1/200)$ i.i.d. entries and is preprocessed via column-normalization, i.e., $\|A_i\|_2 =1$ for any $i$.
The observation $y$ is generated via $y=Ax^*+\epsilon$ with 30 dB noise.
With these settings, the convergence condition for the CCD algorithm becomes $0<\mu<1.$
To justify the effect of the stepsize parameter, we vary $\mu$ from $0$ to $1$,
as well as consider different $q,$ that is, $q=0.1, 0.3, 0.5, 0.7, 0.9.$
The terminal rule of the CCD algorithm is set as either the recovery mean square error (RMSE) $\frac{\|x^{(n)}-x^*\|_2}{\|x^*\|_2}$  less than a given precision $tol$ (in this case, $tol = 10^{-2}$) or the number of iterations more than a given positive integer $MaxIter$ (in this case, $MaxIter = 1.6 \times 10^{5}$).
The regularization parameter $\lambda$ is set as 0.009 and fixed for all experiments.
The experiment results are shown in Fig. {\ref{Effect_mu}}.

From Fig. 1, we can observe that
the stepsize parameter $\mu$ has almost no influence on the recovery quality of the CCD algorithm (as shown in Fig. {\ref{Effect_mu}}(a))
while it significantly affects the time efficiency of the proposed algorithm (as shown in Fig. {\ref{Effect_mu}}(b)).
Basically, we can claim that larger stepsize implies faster convergence.
This coincides with the common sense.
Therefore, in practice, we suggest a larger step-size like $0.95/L_{\max}$ for the CCD algorithm.
However, as shown in Fig. 1, there are some abnormal points when $q=0.1$ and $0.3$ with smaller $\mu.$
More specifically, when $q=0.1$ with $\mu = 0.1, 0.2, 0.3, 0.4$ as well as $q=0.3$ with $\mu = 0.1, 0.2$,
the recovery error and computational time of these cases are much larger than the other cases.
This phenomena is mainly due to that in these cases, the CCD algorithm stops
when the number of iterations achieves to the given maximal number of iterations but not the recovery error reaches to the given recovery precision.
While in the other cases, the CCD algorithm stops when the recovery error attains to the given recovery precision.
Therefore, in these special cases, more iterations are implemented and thus,
more computational time is required
as well as worse recovery quality is obtained,
as compared with those of the other cases.

\begin{figure}[!t]
\begin{minipage}[b]{.49\linewidth}
\centering
\includegraphics*[scale=0.32]{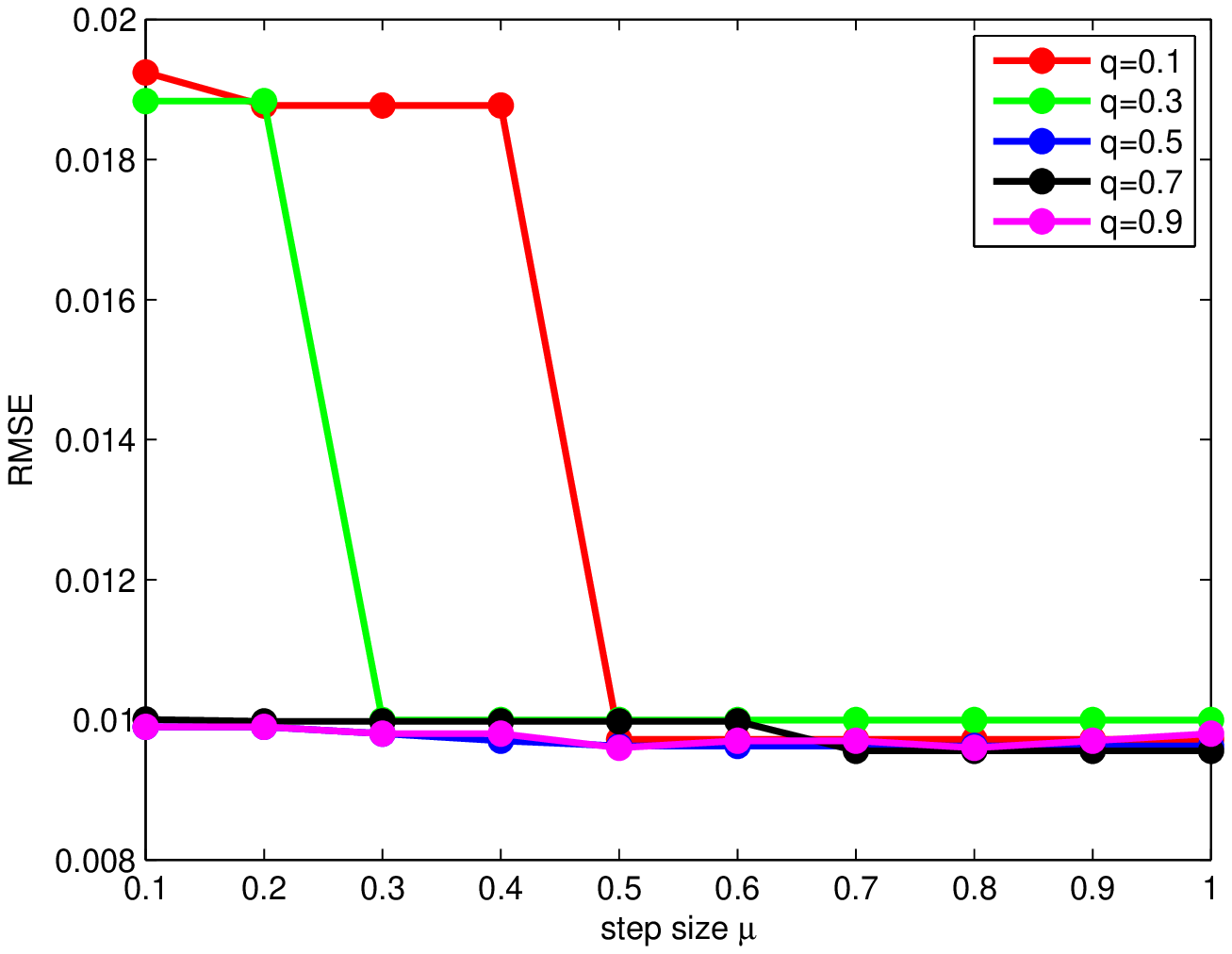}
%  \vspace{-.5cm}
\centerline{{\small (a) Recovery Error}}
\end{minipage}
\hfill
\begin{minipage}[b]{.49\linewidth}
\centering
\includegraphics*[scale=0.32]{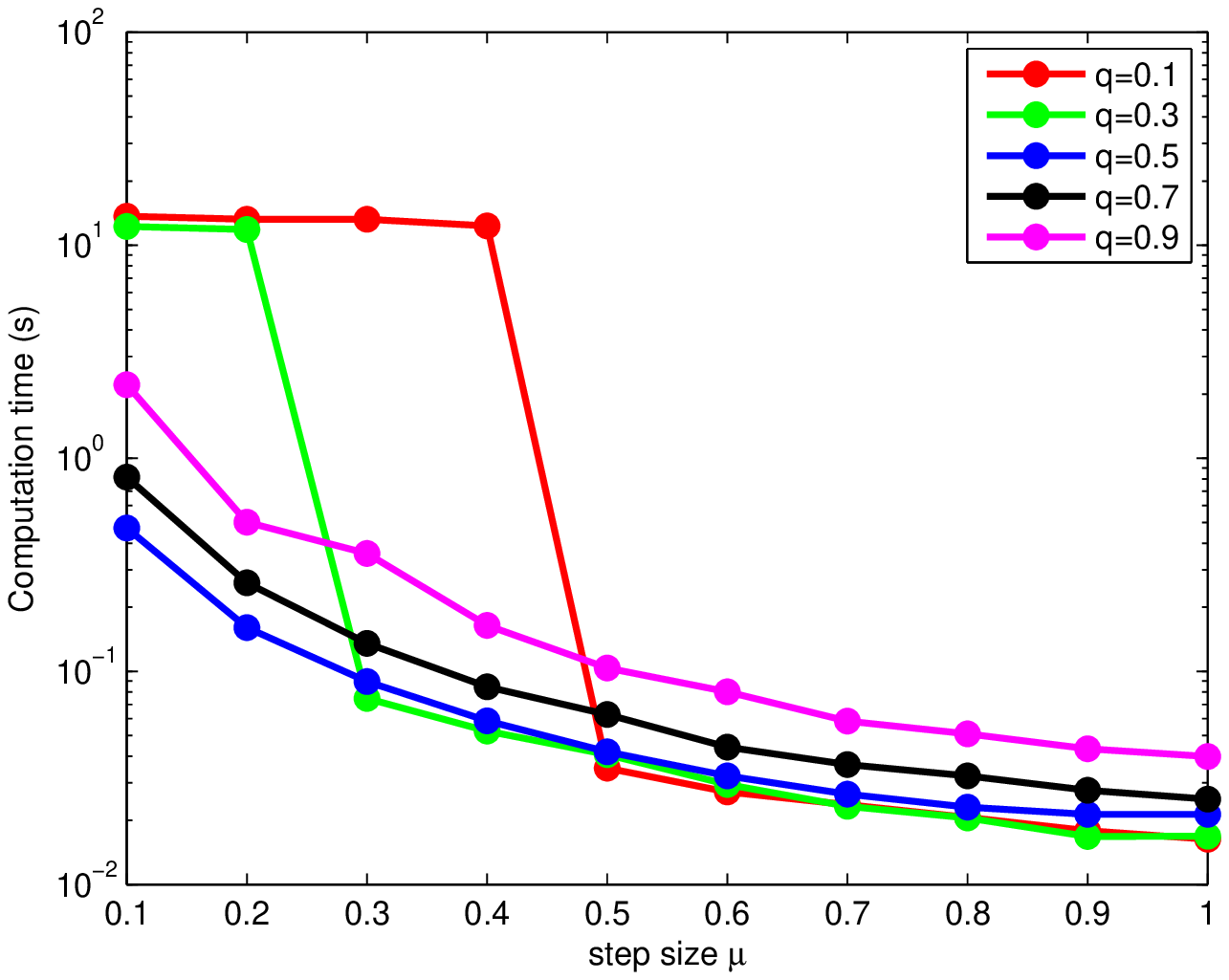}
%  \vspace{-.5cm}
\centerline{{\small (b) Computational Time}}
\end{minipage}
\hfill
\caption{ Experiment for the justification of the effect of stepsize parameter $\mu$ on the performance of the CCD algorithm with different $q$.
(a) The trends of recovery error of the CCD algorithm with different $q$.
(b) The trends of the computational time of the CCD algorithm with different $q$.
}
\label{Effect_mu}
\end{figure}

\section{Conclusion}

We propose a cyclic coordinate descent (CCD) algorithm for the non-convex $l_q$ ($0<q<1$) regularization problem.
The main contribution of this paper is the establishment of the convergence analysis of the proposed CCD algorihm.
In summary, we have verified that
\begin{enumerate}
\item[(i)]
the proposed CCD algorithm converges to a stationary point as long as $0<\mu<\frac{1}{L_{\max}}$ with $L_{\max} = \max_i \|A_i\|_2^2$,
which is weaker than the convergence condition for the iterative jumping thresholding (IJT) algorithm applied to $l_q$ regularization with $0<\mu<\|A\|_2^{-2}$ {\cite{ZengIJT2014}}.
This coincides with the common sense because the CCD algorithm proposed in this paper can be viewed as a Gauss-Seidel type algorithm while IJT algorithm can be seen as a Jocobi type algorithm.

\item[(ii)]
the CCD algorithm further converges to a local minimizer of $l_q$ regularization if the regularization parameter $\lambda$ is relatively small.
\end{enumerate}

Compared with the tightly related work in {\cite{lqCD-Marjanovic-2014}},
there are two significant improvements.
On one hand, we get rid of the column-normalization requirement of the measurement matrix $A$ via introducing a stepsize parameter that improves the flexibility and applicability of the CCD algorithm.
% via adopting such a stepsize parameter.
In addition, the proposed CCD algorithm has almost the same performance of the $l_q$CD algorithm as demonstrated by the numerical experiments.
On the other hand and also the more important one,
we can justify the convergence of the proposed CCD algorithm by introducing the stepsize parameter.
While only the subsequential convergence of $l_q$CD algorithm can be claimed in {\cite{lqCD-Marjanovic-2014}}.

\end{document}